\RequirePackage{fix-cm}

\documentclass[a4paper]{svjour3}

\smartqed  % flush right qed marks, e.g. at end of proof

\usepackage{graphicx}
\usepackage[margin=1in]{geometry} %wen zhang kuan du
\usepackage{epstopdf}
\usepackage{color}
\usepackage{amsmath}
\usepackage{multirow}
\usepackage{algorithm}
\usepackage{algorithmic}
\usepackage{subfigure}
\usepackage{dsfont}
\usepackage{amssymb}
\usepackage[latin1]{inputenc}
\usepackage[colorlinks, linkcolor=blue, anchorcolor=blue, citecolor=blue]{hyperref}

\usepackage{bbding}%sheng cheng tongxin zuozhe xinfeng
\newtheorem{assumption}{Assumption}
\usepackage{setspace}
\usepackage[justification=centering]{caption}

\allowdisplaybreaks[4]

\makeatletter
\newenvironment{breakablealgorithm}
{% \begin{breakablealgorithm}
	\refstepcounter{algorithm}% New algorithm
	\hrule height.8pt depth0pt \kern2pt% \@fs@pre for \@fs@ruled
	\renewcommand{\caption}[2][\relax]{% Make a new \caption
		{\raggedright\textbf{\ALG@name~\thealgorithm} ##2\par}%
		\ifx\relax##1\relax % #1 is \relax
		\addcontentsline{loa}{algorithm}{\protect\numberline{\thealgorithm}##2}%
		\else % #1 is not \relax
		\addcontentsline{loa}{algorithm}{\protect\numberline{\thealgorithm}##1}%
		\fi
		\kern2pt\hrule\kern2pt
	}
}{% \end{breakablealgorithm}
\kern2pt\hrule\relax% \@fs@post for \@fs@ruled
}
\makeatother

\begin{document}
\begin{spacing}{1.5}

\title{A Regularized Limited Memory Subspace Minimization Conjugate Gradient Method for Unconstrained Optimization }
\author{ Wumei Sun$^1$ \and Hongwei Liu$^1$  \and  Zexian Liu$^2$}
%\title{Accelerated\ Stochastic\ Variance\ Reduction\ for\ a\ Class\ of\ Convex\ Optimization\ Problems}

%\title{Accelerated\ proximal\ stochastic\ variance\ reduction\ for\ a\ class\ of\ difference-of-convex\ optimization}

%\author{Lulu He\and Jimin Ye\and Jianwei E}

\institute{
	 Wumei Sun \at
	\email{sunwumei1992@126.com}
	\and	
	Hongwei Liu  \Envelope \at
	\email{hwliuxidian@163.com}
	\and
	Zexian Liu \at
	\email{liuzexian2008@163.com}
	         %  \\
%             \emph{Present address:} of F. Author  %  if needed
	\and
        $^1$  School of Mathematics and Statistics, Xidian University, Xi'an 710126,  China \\
        $^2$  School of Mathematics and Statistics, Guizhou University, Guiyang 550025, China \\
           }

\date{Received: date / Accepted: date}
\maketitle

\begin{abstract}
In this paper, based on the limited memory techniques and subspace minimization conjugate gradient (SMCG)  methods, a regularized limited memory subspace minimization conjugate gradient method is proposed, which contains two types of iterations. In SMCG iteration, we obtain the search direction by minimizing the approximate quadratic model or approximate regularization model. In RQN iteration, combined with regularization technique and BFGS method, a modified regularized quasi-Newton method is used in the subspace to improve the orthogonality. Moreover, some simple acceleration criteria and an improved tactic for selecting the initial stepsize  to enhance the efficiency of the algorithm are designed. Additionally, an generalized nonmonotone line search is utilized and the global convergence of our proposed algorithm is established under mild conditions.  Finally, numerical results show that, the proposed algorithm has a significant improvement over ASMCG\_PR  and is superior to the particularly well-known limited memory conjugate gradient software packages CG\_DESCENT (6.8) and CGOPT(2.0) for the CUTEr library.

\end{abstract}

\section{Introduction}\label{sec1}
Consider problem
	\begin{equation}\label{(1.1)}
		\mathop {\min }\limits_{x \in {\mathbb{R}^n}} f(x),
	\end{equation}
where $f:{\mathbb{R}^n} \to \mathbb{R}$ is a continuously differentiable nonlinear function.

Throughout the article, we use the following notations.  $s_{k-1} = x_k - x_{k-1}$,  $f_k = f(x_k)$, $g_k = g(x_k)$, $y_{k-1} = g_k -g_{k-1} $, $\|\cdot\|$ represents the Euclidean norm and $\lambda_{\max}$ denotes the maximum eigenvalue. Moreover, $\mathrm{dist}\{x,\mathcal{S}\}= \mathrm{inf}\{\|y-x\|, y\in \mathcal{S}\},$ where $x\in \mathbb{R}^n$ and $\mathcal{S}\in \mathbb{R}^n.$

Nonlinear conjugate gradient(CG) method is a well-known method for solving the problem \eqref{(1.1)}, which main iteration is
\begin{equation}\label{(1.2)}
	{x_{k + 1}} = {x_k} + {\alpha _k}{d_k}, \;\; k=0,1,2,\cdots,
\end{equation}
where $x_k$ is the $k$th iteration point, $\alpha _k  > 0$ is the stepsize and $d_k$ is the search direction obtained by
\begin{equation}\label{(1.3)}
d_{0}=-g_{0},\; d_{k}=-g_{k} + {\beta _k}{d_{k-1}}, \; k\geq 1,
\end{equation}
where ${g_{k}}$ is the gradient of $f({x_{k}})$ and ${\beta _k}$ is the conjugate parameter.

It is shown in theory that the convergence and numerical performance variation of different CG methods depend on the selection of conjugate parameters. Some very classical choices of the conjugate parameter $\beta_k$ are Fletcher-Reeves(FR) \cite{Fletcher64}, Polak-Ribi\`{e}re-Polyak(PRP) \cite{Polak69a,Polyak69b}, Dai-Yuan(DY) \cite{Dai99}  and Hestenes-Stiefel(HS) \cite{Hestenes52}, and are given by
\begin{align*}
\beta _k^{FR} = \frac{{{{\left\| {{g_{k + 1}}} \right\|}^2}}}{{{{\left\| {{g_k}} \right\|}^2}}},~~\beta _k^{PRP} = \frac{{g_{k + 1}^T{y_k}}}{{{{\left\| {{g_k}} \right\|}^2}}},~~\beta _k^{DY} = \frac{{{{\left\| {{g_{k + 1}}} \right\|}^2}}}{{d_k^T{y_k}}},~~\beta _k^{HS} = \frac{{g_{k + 1}^T{y_k}}}{{d_k^T{y_k}}}.
\end{align*}
CG algorithms have evolved considerably, and some well-known CG  packages such as CG\_DESCENT \cite{Hager05,Hager06b}  and CGOPT \cite{Dai13} have been proposed in recent years. Other recent related studies on nonlinear CG algorithms can be found in \cite{Dai11,Hager06a}.	

The subspace minimization conjugate gradient (SMCG) algorithm, as a generalization of the CG algorithm, has received much attention from scholars \cite{Andrei14,Yang17}, which can be traced back to the work of Yuan and Stoer \cite{Yuan95}. The search direction of SMCG method is obtained by minimizing the following problem:
\begin{equation}\label{(1.4)}
\min\limits_{d \in \Omega_{k}} \ \   g_{k }^Td + \frac{1}{2}{d^T}{{B}_{k}}d,
\end{equation}
where $\Omega_{k}$ is a subspace spanned by the vectors $g_{k}$ and $s_{k-1},$ i.e., $\Omega_{k}=Span\{g_{k},s_{k-1}\} ,$ and ${{B}_{k}} \in \mathbb{R}^{n \times n}$ is an approximation of Hessian matrix, which is  positive definite and symmetric. Then the search direction $d$  is given by
\begin{equation}\label{(1.5)}
	{d} = u {g_{k}} + v {s_{k-1}},
\end{equation}
where $u$ and $v$ are both real parameters. Substituting \eqref{(1.5)} to \eqref{(1.4)} and combined with the standard secant equation $ {B}_{k}s_{k-1}=y_{k-1},$  formula \eqref{(1.4)} is reorganized as follows:
\begin{equation}\label{(1.6)}
\min\limits_{u,v \in \mathbb{R} } {\left( \begin{array}{c}
                                     \| g_{k} \|^{2} \\
                                     g_{k}^{T}s_{k-1}
                                   \end{array}
\right )} ^{T} {\left( \begin{array}{c}
                         u \\
                         v
                       \end{array}\right )} + \frac{1}{2}{\left( \begin{array}{c}
                         u \\
                         v
                       \end{array}\right )}^{T} {\left(\begin{array}{cc}
                                                         \rho_{k} &  g_{k}^{T}y_{k-1} \\
                                                         g_{k}^{T}y_{k-1} & s_{k-1}y_{k-1}
                                                       \end{array}\right )}{\left( \begin{array}{c}
                         u \\
                         v
                       \end{array}\right ).}
\end{equation}
where $\rho_{k} \approx g^{T}_{k}{B}_{k}g_{k}$.

On the basis of the Barzilai-Borwein(BB) method \cite{Barzilai88}, Dai and Kou \cite{Dai16} proposed an effective BBCG3 method for strictly convex quadratic minimization problem. Afterwards, based on BBCG3 method, Liu and Liu \cite{Liu19} proposed SMCG\_BB method for solving general unconstrained optimization problems.
Motivated by SMCG\_BB method, some efficient SMCG methods \cite{Li18,Li19,Wang19,Zhao21} were later proposed, among which the method based on the regularization model presented by Zhao et al. \cite{Zhao21} is the best in the numerical performance.

The nonlinear CG method is very effective for unconstrained optimization problems. However, the convergence of the algorithm can be very slow for some ill-posed problems and even for quadratic problems with very small dimensions, which may be due to the loss of orthogonality \cite{Hager2013}. Hager and Zhang \cite{Hager2013} pointed out theoretically that the generated successive gradients either in the CG method or the L-BFGS method for the quadratic test problem should be orthogonal. Yet, Hager and Zhang \cite{Hager2013} observed that, when solving the quadratic strictly convex  minimization problem PALMER1C in the CUTEr library \cite{Gould03},  the CG method loses orthogonality due to the rounding errors, while L-BFGS method preserves the orthogonality. In view of this, they developed the limited memory CG method (CG\_DESCENT(6.8)) to correct the possible loss of orthogonality in ill conditioned optimization problems. For the test problems in the CUTEr library \cite{Gould03}, their performance results indicated that CG\_DESCENT(6.8) has an significant improvement over their previously proposed package CG\_DESCENT(5.3).

Although CG\_DESCENT(6.8) \cite{Hager2013} is an efficient method for unconstrained optimization, it still suffers from the following shortcomings: \\
(i) In the numerical implementation, the AWolfe line search \cite{Hager06b} utilized in the algorithm CG\_DESCENT(6.8) does not guarantee global convergence. \\
(ii) CG\_DESCENT(6.8) contains the following three pre-conditioners, corresponding to three different iterations:
\begin{equation}\label{(1.7)}
P_k = I,\; P_k = Z_k \hat{B}^{-1}_{k+1}Z^T_k,\; P_k = Z_k \hat{B}^{-1}_{k+1}Z^T_k +\sigma_k \bar{Z}_k \bar{Z}^T_k,
\end{equation}
where $\sigma_k$ is determined by (4.2) of \cite{Hager2013}, $\hat{B}_{k+1}$, $Z_k$ and $\bar{Z}_k$ are given by the matrices in literature \cite{Hager2013}. These three pre-conditioners make the algorithm CG\_DESCENT(6.8) look complex.\\
(iii) In the convergence analysis, the algorithm CG\_DESCENT(6.8) needs to impose the following assumptions on the pre-conditioners:
\begin{equation}\label{(1.8)}
\|P_k\| \leq \gamma_0,\; g^T_{k+1}P_k g_{k+1} \geq \gamma_1 \|g_{k+1}\|^2, \;d^T_k P^{-1}_k d_k \geq \gamma_2 \|d_{k}\|^2,
\end{equation}
where $\gamma_0 >0,$ $\gamma_1 >0$ and $\gamma_2 >0$. These assumptions are comparatively strict and difficult to be verified in actual practice.

To address the above-mentioned shortcomings, Liu et al. \cite{Liu20} presented an improved Dai¨CKou CG algorithm called CGOPT(2.0), which combines limited memory technology and Dai-Kou CG method.  In CGOPT(2.0) \cite{Liu20}, they utilized a modified quasi-Newton method to restore the lost orthogonality, and established the  convergence of CGOPT(2.0) with fewer assumptions.  Some numerical experiments indicated that CGOPT(2.0)  is better than the famous  CG software package CG\_DESCENT(6.8) \cite{Hager2013}.

In view of the above discussion, a regularized limited memory subspace minimization conjugate gradient method on the basis of SMCG method and limited memory technique is studied in this paper. To recover orthogonality, we propose a modified regularized quasi-Newton method. The major contributions of this paper are the following.
\begin{enumerate}
  \item A regularized limited memory subspace minimization conjugate gradient algorithm is proposed, which combines limited memory technology and SMCG method.
  \item Based on the idea of regularization and  BFGS method, an improved regularized  quasi-Newton method is exploited to improve orthogonality.
  \item Some simple acceleration criteria and an improved initial stepsize selection strategy are designed to enhance the efficiency of the algorithm. Additionally, an generalized nonmonotone line search condition is presented, which may be regarded as an extension of the Zhang-Hager's \cite{Zhang04} nonmonotone line search.
  \item The convergence of the method is built under mild conditions and the corresponding numerical performance shows that the new method is much more effective than the existing methods.
\end{enumerate}

The structure of the paper is as follows.  In Section \ref{sec2}, we describe the detail of the regularized limited memory subspace minimization conjugate gradient algorithm, including the direction selection of SMCG iteration and regularized Quasi-Newton iteration and an effective acceleration technique. Moreover, the decision of the initial step size and the generalized nonmonotone Wolfe line search are also given in this section. In Section \ref{sec3},  some important properties of the search direction are analyzed and the global convergence of the proposed algorithm is established. Numerical experiments for algorithm comparison are showed in Section \ref{sec4}. Conclusions are given in the last section.

\section{A Regularized Limited Memory Subspace Minimization Conjugate Gradient Algorithm}\label{sec2}
In the section, combining the idea of subspace minimization and regularization  quasi-Newton method, we present a regularized limited memory subspace minimization conjugate gradient algorithm.
Firstly, we give the choices of search direction under different iterations. Subsequently, we develop a very effective acceleration technique, a modified initial step selection strategy and generalized nonmonotonic line search technology to optimize the performance of the proposed algorithm. Finally, the details of algorithm RL\_SMCG are described.

\subsection{Direction Selection of SMCG Iteration and Regularized Quasi-Newton Iteration}
The regularized limited memory subspace minimization conjugate gradient method mainly contains two kinds of iterations which are SMCG iteration and regularized quasi-Newton(RQN) iteration, respectively. Furthermore, the search direction derivation of the two iterations is also different.
\subsubsection{SMCG iteration }
The search direction selection of SMCG iteration is closely related to the properties of the objective function $f(x)$ at the iteration point $x_k$. By reference \cite{Dai02,Yuan91}, defined
\begin{equation}\label{(2.1)}
{t_k} = \left| {{{2\left( {{f_{k - 1}} - {f_k} + g_k^T{s_{k - 1}}} \right)}} \left/ \right.{\left({s_{k - 1}^T{y_{k - 1}}}\right)} - 1} \right|,
\end{equation}
to describe  how $f(x)$ approaches a quadratic function on a line segment between $x_{k-1}$ and $x_{k}$.
Literature \cite{Liu18} indicates that if the condition
\begin{equation}\label{(2.2)}
{t_k} \le {\bar{\xi}_4}\;\;{\rm{or}}\;\;\left( {{t_k} \le \bar{\xi}_5\;\;{\rm{and}}\;\;{t_{k - 1}} \le \bar{\xi}_5} \right),
\end{equation}
is satisfied, where $\bar{\xi}_4$ and $\bar{\xi}_5$ are the smaller positive constants and $\bar{\xi}_4<\bar{\xi}_5$, $f(x)$ may be near to a quadratic function on a line between ${x_{k - 1}}$ and ${x_k}$. Moreover, According to \cite{Sun21}, we know that  if the following condition
\begin{equation}\label{(2.3)}
{\bar{\xi}_1} \le \frac{{s_{k - 1}^T{y_{k - 1}}}}{{{{\left\| {{s_{k - 1}}} \right\|}^2}}} \le \frac{{{{\left\| {{y_{k - 1}}} \right\|}^2}}}{{s_{k - 1}^T{y_{k - 1}}}} \le {\bar{\xi}_2},
\end{equation}
is satisfied, then the condition number of the Hessian matrix of the normal function may be not very large, here $\bar{\xi}_1$ and $\bar{\xi}_2$ are positive constants.

Similar to \cite{Zhao21}, based on some certain properties of the function $f(x)$ at the current point $x_{k}$, we derive different search direction by dividing it into the following four cases.

\textbf{(i)} If the condition \eqref{(2.3)} is satisfied while the condition \eqref{(2.2)} are not, this implies that the quadratic model may not be able to approach the objective function $f(x)$ well at the present iteration point $x_{k}$. Then, search direction $d_k$ will be obtained by minimizing the following cubic regular subproblem, i.e.
\begin{equation}\label{(2.4)}
\mathop {\min }\limits_{{d_k} \in {\Omega _k}} {m_k}\left( {{d_k}} \right) = d_k^T{g_k} + \frac{1}{2}d_k^T{{B}_k}{d_k} + \frac{1}{3}{\sigma _k}\left\| {{d_k}} \right\|_{{{B}_k}}^3,
\end{equation}
where $\Omega_{k}$ is a subspace spanned by the vectors $g_{k}$ and $s_{k-1},$  ${{B}_{k}} \in \mathbb{R}^{n \times n}$ is an approximation of Hessian matrix, which is  positive definite and symmetric and satisfying the secant condition ${{B}_k}{s_{k - 1}} = {y_{k - 1}},$ ${\sigma _k}\geq 0$ is an adaptive regularization parameter obtained from interpolation condition and $d_{k}$ is determined by
\begin{equation}\label{(2.5)}
{d}_k = u_k {g_{k}} + v_k {s_{k-1}},
\end{equation}
where $v_k$ and  $u_k$ are parameters to be established. Obviously, we could obtain \eqref{(2.4)} by giving \eqref{(1.4)} a weighted regularization term $\frac{1}{3}{\sigma _k}\left\| {{d_k}} \right\|_{{{B}_k}}^3$. Substituting \eqref{(2.5)} to \eqref{(2.4)}, it is easy to obtain that \eqref{(2.4)} is equivalent to
\begin{equation}\label{(2.6)}
\min\limits_{u_k,v_k \in R } {\left( \begin{array}{c}
                                     \| g_{k} \|^{2} \\
                                     g_{k}^{T}s_{k-1}
                                   \end{array}
\right )} ^{T} {\left( \begin{array}{c}
                         u_k \\
                         v_k
                       \end{array}\right )} + \frac{1}{2}{\left( \begin{array}{c}
                         u_k \\
                         v_k
                       \end{array}\right )}^{T} {\bar{B}_k}
                                                       {\left( \begin{array}{c}
                         u_k \\
                         v_k
                       \end{array}\right)}  + \frac{\sigma _k}{3}\left\| {\left( \begin{array}{c}
                         u_k \\
                         v_k
                       \end{array}\right)}  \right\|^{3}_{\bar{B}_k}.
\end{equation}
where $\bar{B}_k =\left(\begin{array}{cc}
                                                         \rho_{k} &  g_{k}^{T}y_{k-1} \\
                                                         g_{k}^{T}y_{k-1} & s_{k-1}y_{k-1}
                                                       \end{array}\right )$ is a positive definite and symmetric  matrix,
$\rho_{k}$ is an estimate of $g^{T}_{k}{B}_{k}g_{k}$. Similar to BBCG3 \cite{Dai16}, we also use $ { \frac{3}{2}}\frac{{{{\left\| {{y_{k-1}}} \right\|}^2}}}
{{s_{k-1}^T{y_{k-1}}}}I$ to estimate $B_k$ in the term $\rho _{k} $, which means $\rho _{k} = { \frac{3}{2}}\frac{{{{\left\| {{y_{k-1}}} \right\|}^2}}}
{{s_{k-1}^T{y_{k-1}}}}{\left\| {{g_{k}}} \right\|^2}$. Then, by solving problem \eqref{(2.6)} we obtain the following solutions about $u_k$ and $v_k$:
\begin{equation}\label{(2.7)}
\left(
  \begin{array}{c}
    {u}_k \\
    {v}_k \\
  \end{array}
\right)
= \left(
    \begin{array}{c}
      \frac{1}{{\left( {1 + {\sigma _k}{{\left( {{\varpi^*}} \right)}}} \right){\Delta _k}}}\left( {g_k^T{y_{k - 1}}g_k^T{s_{k - 1}} - s_{k - 1}^T{y_{k - 1}}{{\left\| {{g_k}} \right\|}^2}} \right) \\
      \frac{1}{{\left( {1 + {\sigma _k}{{\left( {{\varpi^*}} \right)}}} \right){\Delta _k}}}\left( {g_k^T{y_{k - 1}}{{\left\| {{g_k}} \right\|}^2} - {\rho _k}g_k^T{s_{k - 1}}} \right)\\
    \end{array}
  \right),
\end{equation}
among them,
\begin{equation}\label{(2.8)}
\Delta _k = {\left|\begin{array}{cc}
                                                         \rho_{k} &  g_{k}^{T}y_{k-1} \\
                                                         g_{k}^{T}y_{k-1} & s_{k-1}y_{k-1}
                                                       \end{array}\right |} = \rho_{k}s_{k-1}y_{k-1} -(g_{k}^{T}y_{k-1})^2 >0,
\end{equation}
$\sigma _k$ and $\varpi^*$ are the same as those in literature \cite{Zhao21}, which will not be repeated here.

\textbf{(ii)}  If both  conditions \eqref{(2.3)} and \eqref{(2.2)} hold, this indicates that the objective function $f(x)$ may approach the quadratic model at the current iteration point $x_k.$  Since that is the case, let $\sigma_k =0$, i.e. we consider deriving the search direction by solving the minimization problem \eqref{(1.6)}. Like \textbf{(i)}, we choose $\rho _{k} = { \frac{3}{2} }\frac{{{{\left\| {{y_{k-1}}} \right\|}^2}}}
{{s_{k-1}^T{y_{k-1}}}}{\left\| {{g_{k}}} \right\|^2}$ and $\Delta_{k}$ is determined by \eqref{(2.8)},
then we obtain the following unique solution of quadratic approximate problem \eqref{(1.6)}:
\begin{equation}\label{(2.9)}
\left(
  \begin{array}{c}
    \bar{u}_k \\
    \bar{v}_k \\
  \end{array}
\right)
= \left(
    \begin{array}{c}
      \frac{1}{\Delta_{k}} ( g^{T}_{k}y_{k-1} g^{T}_{k}s_{k-1} - s^{T}_{k-1}y_{k-1}\|g_{k}\|^{2} ) \\
      \frac{1}{\Delta_{k}} ( g^{T}_{k}y_{k-1}\|g_{k}\|^{2}-\rho_{k}g^{T}_{k}s_{k-1}  )\\
    \end{array}
  \right),
\end{equation}
here the search direction is calculated by $d_k = \bar{u}_{k} {g_{k}} + \bar{v}_{k} {s_{k-1}}$, where $\bar{u}_{k}$ and $\bar{v}_{k}$ are determined by \eqref{(2.9)}.

\textbf{(iii)} If  condition \eqref{(2.3)} is not satisfied  and the conditions
\begin{equation}\label{(2.10)}
{{\left| {g_k^T{y_{k - 1}}g_k^T{s_{k - 1}}} \right|}} \le {\bar{\xi}_3}{{s_{k - 1}^T{y_{k - 1}}{{\left\| {{g_k}} \right\|}^2}}}  {\rm{  \;\;and\;\;  }}  {{s_{k - 1}^T{y_{k - 1}}}} \geq {\bar{\xi}_1}  {{{{\left\| {{s_{k - 1}}} \right\|}^2}}},
\end{equation}
are satisfied, where $ 0 \leq \bar{\xi}_3 \leq 1 $, the condition number of the Hessian matrix may be lager, hence the search direction obtained in cases \textbf{(i)} and \textbf{(ii)} may not be better. However, the condition \eqref{(2.10)} can ensure sufficient descent and linear growth in HS conjugate gradient method. Moreover, because of the finite termination nature of the HS conjugate gradient method for solving exact convex quadratic minimization problems, this choice of direction allows for faster convergence of the algorithm. Then, in this case, the search direction is determined by \eqref{(1.3)} and $\beta_k = \beta^{HS}_k$.

\textbf{(iv)} If neither condition \eqref{(2.3)} nor \eqref{(2.10)} holds,  then we pick the following direction, i.e. :
\begin{equation}\label{(2.11)}
{d_k} = -{g_k}.
\end{equation}

In summary, the search direction in the SMCG iteration can be described as in the following:
\begin{equation}\label{(2.12)}
{d_k} = \left\{ \begin{array}{l}
 {u}_k {g_k} +  {v}_k {s_{k - 1}},\;\;\;\;\;\;\;\text{if}\;\eqref{(2.3)}\;\text{holds and}\; \eqref{(2.2)}\;\text{does not hold},\\
 \bar{u}_k {g_k} + \bar{v}_k {s_{k - 1}},\;\;\;\;\;\;\;\text{if}\;\eqref{(2.3)}\;\text{holds and}\; \eqref{(2.2)}\;\text{holds},\\
- {g_k} + \beta _k^{HS}{d_{k - 1}},\;\;\;\;\;\text{if}\;\eqref{(2.3)}\;\text{does not hold and}\;\eqref{(2.10)}\;\text{holds},\\
\;\;\;\;\;\;\;- {g_k},\;\;\;\;\;\;\;\;\;\;\;\;\;\;\;\;\;\;\text{if neither}\;\eqref{(2.3)}\;\text{nor}\;\eqref{(2.10)}\;\text{holds},
\end{array} \right.
\end{equation}
where $ {u}_{k}$ and $ {v}_{k}$ are determined by  \eqref{(2.7)};
$\bar{u}_{k}$ and $\bar{v}_{k}$ are determined by \eqref{(2.9)}.

If the successive gradients have orthogonality or the lost orthogonality is restored, the algorithm performs SMCG iteration. On the contrary, if the orthogonality is lost, the iteration will turn to the following regularized quasi-Newton iteration to improve the orthogonality.

\subsubsection{Regularized Quasi-Newton(RQN) iteration }
When the successive gradients lose their orthogonality, the iteration switches from SMCG iteration to RQN iteration. In other words, a modified regularized BFGS algorithm in subspace $\mathcal{S}_k$ is proposed to restore the orthogonality, where $\mathcal{S}_k$  is a subspace generated by the following limited memory $m$ search directions
\[
\mathcal{S}_k = span \left\{ d_{k-1}, d_{k-2},\cdots, d_{k-m} \right\},
\]
where $m>0$ and $m$ is the number of limited memory. In this article, the limited memory $m$ selected in our algorithm does not exceed 11. Then, as soon as orthogonality is corrected, the RQN iteration is terminated and the SMCG iteration is triggered immediately.

First, we introduce some preparations for turning to RQN iteration.
Let ${S}_k \in \mathbb{R} ^{n\times m}$  be  a matrix  which has columns consisting of $d_{k-1}, d_{k-2},\cdots, d_{k-m}$. In similar fashion to limited memory CG method \cite{Hager2013}, we also assume that columns of ${S}_k$ are line-independent.  Let the QR factorization of ${S}_k$ be ${S}_k = Z_k \bar{R}_k$, where the columns of $ {Z}_k \in \mathbb{R} ^{n\times m}$ form the normal orthogonal bases for  subspace $\mathcal{S}_k$ and $ {\bar{R}}_k \in \mathbb{R} ^{m\times m}$ is the upper triangular matrix with positive diagonal terms.

If $g_k$ is included almost in subspace $\mathcal{S}_k$, then we think that the orthogonality property of the algorithm may be lost.  In this case, we interrupt the SMCG iteration and  move to minimize the objective function in the subspace $\mathcal{S}_k$:
\begin{equation}\label{(2.13)}
\mathop {\min }\limits_{z \in {\mathcal{S}_k}} f(x_k +z).
\end{equation}
The solution to the subspace problem \eqref{(2.13)} will  improve the orthogonality and guide us to a suitable search direction that will lead us out of the subspace $\mathcal{S}_k$. Similar to \cite{Hager2013}, we utilize the distance from $g_k$ to subspace $\mathcal{S}_k$ to judge whether orthogonality is lost. If the condition
\begin{equation}\label{(2.14)}
\mathrm{dist} \left\{ g_k,\mathcal{S}_k \right\} \leq \tilde{\eta}_0 \|g_k\|
\end{equation}
is satisfied, where $0<\tilde{\eta}_0<1$ and $\tilde{\eta}_0$ is small, we think $g_k$ is almost contained in $\mathcal{S}_k$, it means that the orthogonality of the successive gradients has lost. Then, we switch to RQN iteration to solve the subspace problem \eqref{(2.13)} until the gradient is nearly orthogonal enough to the subspace to meet the condition
\begin{equation}\label{(2.15)}
\mathrm{dist} \left\{ g_k,\mathcal{S}_k \right\} \geq \tilde{\eta}_1 \|g_k\|,
\end{equation}
where $0 < \tilde{\eta}_0 < \tilde{\eta}_1 < 1$. At this time, the algorithm iteration will go away subspace $\mathcal{S}_k$ and turn to the SMCG iteration.  Because the column of $Z_k$ is the orthonormal basis of $\mathcal{S}_k$, it's not hard to know from the
definition of $\mathrm{dist} \left\{ g_k,\mathcal{S}_k \right\}$ that \eqref{(2.14)} and \eqref{(2.15)} can be expressed as
\begin{equation}\label{(2.16)}
\left( 1-\tilde{\eta}^2_0 \right)\|g_k\|^2  \leq  \left\| Z^T_k g_k\right\|^2,
\end{equation}
and
\begin{equation}\label{(2.17)}
\left( 1-\tilde{\eta}^2_1 \right)\|g_k\|^2  \geq  \left\| Z^T_k g_k\right\|^2.
\end{equation}
In \cite{Hager2013}, Hager and Zhang utilized the limited memory BFGS (L-BFGS) \cite{Liu1989,Nocedal1980} method to solve the subspace problem \eqref{(2.13)} for restoring the orthogonality, and achieved better numerical results. However, it should be noted that the convergence analysis of the limited memory CG method \cite{Hager2013} requires imposing strict assumptions \eqref{(1.8)} on the preprocessors \eqref{(1.7)}.
Because the dimension $m$ of the chosen subspace $\mathcal{S}_k$ is usually small and when orthogonality is lost, the properties of the function at the iteration point maybe not very good. Based on these, we consider a regularized L-BFGS method in the subspace $\mathcal{S}_k$ for solving the subproblem \eqref{(2.13)}.

The search direction of general quasi-Newton method \cite{Yuan99} for  unconstrained optimization \eqref{(1.1)} is the form of $d_k = -B^{-1}_k g_k$, where $B_k$ is a positive definite and symmetric  approximation to the Hessian matrix. As one of the most popular methods of quasi-Newton method, L-BFGS method stores the approximate Hessian matrix of the objective function using small memory and computes the search direction $d_k$ using the nearest $m$ vector pairs of $\left( s_{k-i}, y_{k-i}\right)$, $i=0, 1, \ldots, m-1.$

Ueda and Yamashita \cite{Ueda2010} presented a regularized Newton method for nonconvex unconstrained optimization, whose search direction $d_k$  is obtained by  solving the following linear equations:
\begin{equation}\label{(2.18)}
\left( \nabla^2 f(x_k)+ \mu I \right)d_k=-\nabla f(x_k),
\end{equation}
where $\mu >0$ is referred to as the regularized parameter. The regularized Newton method \cite{Ueda2010} generally defaults to a step size of 1, and global convergence is guaranteed by controlling the parameter $\mu_k$. However, as a type of Newton method, the regularized Newton method in \cite{Ueda2010} must solve the Hessian matrix of $f$ which is particularly computationally complex.  To address this drawback, some scholars proposed the regularized limited memory BFGS-type method \cite{Tarzangh2015,Liu2014}  for solving unconstrained optimization problems, i.e. the search direction $d_k$ is the solution of the following equations
\begin{equation}\label{(2.19)}
\left( B_k + \mu I \right)d_k=-\nabla f(x_k),
\end{equation}
where matrix $B_k$ is an approximate Hessian determined by a particular quasi-Newton method. Regularization technology can effectively improve the efficiency of quasi-Newton method in solving ill-conditioned problems. Nevertheless, when computing $B_k$ by the L-BFGS method, it is very hard to calculate $\left(B_k + \mu I \right)^{-1}$. Hence, motivated by \cite{Tankaria2022}, we present a regularized quasi-Newton method which combines the BFGS method with the regularized technique to improve orthogonality in the m-dimensional subspace $\mathcal{S}_k$. In this paper, we consider $ B_k + \mu I $ as an approximation of $ \nabla^2 f(x_k)+ \mu I $. Because the matrix $B_k$ is the approximate Hessian of $f(x_k)$ and  $ B_k + \mu I $  can be used as an approximate Hessian of $f(x_k)+\frac{\mu}{2}\|x\|^2$.  At this point, we utilize $\left(s_k,y_k(\mu)\right)$ instead of $(s_k,y_k)$, where
$$ y_k(\mu)=\left( \nabla f(x_{k+1})+ \mu x_{k+1} \right) - \left( \nabla f(x_{k})+ \mu x_{k} \right)  = y_k + \mu s_k.$$
Note that the regularized BFGS method stores as many vector pairs as the traditional BFGS method and hence it does not require additional memory.

In \cite{Li01}, a effective BFGS quasi-Newton method for solving nonconvex unconstrained minimization was proposed by Li and Fukushima  \cite{Li01}, in which the matrix $B_{k+1}$ is updated by
 \[
 B_{k+1}=\left\{ \begin{array}{l}
B_k - \frac{B_k s_k s^T_k B_k}{s^T_k B_ks_k} + \frac{y_k y^T_k}{s^T_k y_k},{\rm{\;\;\;if \; }} \frac{s^T_k y_k}{\|s_k\|^2}>\upsilon \|g_k\|^\alpha,\\
B_k,     {\rm{  \quad  \quad\quad \quad \quad\quad \quad \quad\;\; \;\;\;\; otherwise \; }},
\end{array} \right.
 \]
where $\upsilon >0$ and $\alpha>0$. Some recent advances about modified BFGS method can be found in \cite{Li1999,Gu2003,Tankaria2022}.

Inspired by the quasi-Newton methods described above, we propose an improved regularized BFGS method to solve the subproblem \eqref{(2.13)} in subspace $\mathcal{S}_k$.

\textbf{\emph{Remark 1}.}  In what follows, the variables with hats belong to subspace $\mathcal{S}_k$ , distinguished from the ones found in the full space $\mathbb{R}^n$.

Let $\hat{x}= \left(\hat{x}_1, \hat{x}_2, \cdots, \hat{x}_m, \right)^T \in \mathbb{R}^m$. The subproblem \eqref{(2.13)} can be expressed as
\begin{equation}\label{(2.20)}
\mathop {\min }\limits_{\hat{x} \in \mathbb{R}^m} f(x_k + \hat{x}_1 d_{k-1}+ \hat{x}_2 d_{k-2} +\cdots + \hat{x}_m d_{k-m} ).
\end{equation}
Similar to \cite{Liu20}, because the regularized quasi-Newton directions in the subspace $\mathcal{S}_k$  always transform to the full space $\mathbb{R}^n$ and QR decomposition of matrix $S_k$, we can obtain $d_k=Z_k \hat{d}_k,$ $\hat{g}_k=Z^T_k g_k,$ $\hat{y}_k=Z^T_k y_k,$ $\hat{s}^T_k \hat{y}_k = s^T_k y_k,$  $\|\hat{s}_k\|^2 = \|s_k\|^2$ and $\hat{f}_k=f_k$.

Let $B_k(\mu)=B_k + \mu I $, then inspired by Li and Fukushima  \cite{Li01}, we develop an improved regularized BFGS method to solve the above subproblem \eqref{(2.20)} with a search direction of the form
\begin{equation}\label{(2.21)}
\hat{d}_{k+1} = -\hat{B}^{-1}_{k+1}(\mu) \hat{g}_{k+1},
\end{equation}
where $\hat{B}_{k+1}(\mu)$ is given by
\begin{equation}\label{(2.22)}
 \hat{B}_{k+1}(\mu)=\left\{ \begin{array}{l}
\hat{B}_{k}(\mu) - \frac{\hat{B}_{k}(\mu) \hat{s}_k \hat{s}^T_k \hat{B}_{k}(\mu)}{\hat{s}^T_k \hat{B}_{k}(\mu)\hat{s}_k} + \frac{\hat{y}_k(\mu) \hat{y}^T_k(\mu)}{\hat{s}^T_k \hat{y}_k(\mu)},{\rm{\;\;\;if \; mod }} {(k, l)\neq 0} \; {\rm{and}}\;   \frac{\hat{s}^T_k \hat{y}_k(\mu)}{\hat{s}^T_k\hat{s}_k }\geq\upsilon,\\
\hat{I},     {\rm{ \quad \quad \quad \quad \quad \quad  \quad \quad  \quad\quad \quad \quad\quad \quad \quad \; \;\;\;\; otherwise \; }},
\end{array} \right.
\end{equation}
where  $\upsilon >0$,  $\mathrm{mod} (k, l)\neq 0$ represents the remainder for $k$ modulo $l$,  $\hat{y}_k(\mu) = \hat{y}_k + \mu \hat{s}_k$ and $\mu>0 $ is an important regularized parameter. The condition $\mathrm{mod} (k, l)\neq 0$  means the matrix $\hat{B}_{k}(\mu)$ will be reset to the identity matrix $\hat{I} $ after updating $l$ times, which ensures the good convergence of the algorithm. In the paper, we set $l=\max (m^2, 20)$.  Obviously, $\hat{s}^T_k \hat{y}_k(\mu) >0$, and  as soon as the matrix  $\hat{B}_{k}(\mu)$ is symmetric and positive definitive, it is not hard to prove that the matrix $\hat{B}_{k+1}(\mu)$ is symmetric and positive definitive.

As a very important regularization parameter, $\mu$ is closely related to the convergence analysis of the regularized BFGS method. In this paper, the idea of the trust-region radius is used to find the suitable search direction by controlling $\mu$, in other words,
The ratio of objective function value reduction to model function value reduction is utilized. Then, give the definition of a ratio function $r_k(\hat{d}_k, \mu)$ as follows
\begin{equation}\label{(2.23)}
r_k(\hat{d}_k, \mu) = \frac{\hat{f}(x_k)-\hat{f}(x_k + \alpha_k \hat{d}_k)}{\hat{f}(x_k) - \hat{q}_k(\hat{d}_k, \mu)},
\end{equation}
where $\hat{q}_k : \mathbb{R}^m \times \mathbb{R} \rightarrow \mathbb{R}$ is a function of the form
\begin{equation}\label{(2.24)}
\hat{q}_k(\hat{d}_k, \mu) = \hat{f}(x_k) + \alpha_k \hat{g}^T_k \hat{d}_k + \frac{1}{2}\alpha_k ^2 \hat{d}^T_k \hat{B}_k (\mu) \hat{d}_k.
\end{equation}
Then, if the ratio function $r_k(\hat{d}_k, \mu)$ is relatively large, this means that compared with the reduction of the model function, the reduction of the objective function is large enough, we choose to reduce  the parameter $\mu$. On the flip side, if the ratio function $r_k(\hat{d}_k, \mu)$ is relatively small, i.e., $\hat{f}(x_k)-\hat{f}(x_k + \alpha_k \hat{d}_k)$ is small, we will increase $\mu$. In addition, to ensure that the algorithms converge well, we limit $\mu$ to an interval, i.e. $0<\mu_{\min}<\mu<\mu_{\max}$. In general, if the next iteration point is closer to the current iteration point, the reduction of the function value may not be obvious. At this time, we hope to get a new iteration point by modifying the search direction, then the search direction improved by regular parameter $\mu$ may be a good choice.
Therefore, if $\|\hat{s}_k\|^2 \leq \hat{\tau} $ $(\hat{\tau}>0)$,  our choice and update of $\mu$ are as follows:
\begin{equation}\label{(2.25)}
\mu_{k+1}= \left\{\begin{array}{l}
                    \max \left\{\mu_{\min}, \sigma_1 \mu_k\right\}, \;\;\;\;\mathrm{if}\;r_k(\hat{d}_k, \mu)\geq \sigma_3,\\
                    \min \left\{\mu_{\max}, \sigma_2 \mu_k\right\}, \;\;\;\; \mathrm{otherwise},
                  \end{array}
 \right.
\end{equation}
where $0<\sigma_1\leq1$, $\sigma_2>1$ and $0<\sigma_3\leq1$.
Otherwise, we choose $\mu = 0$, i.e., the regularized BFGS method is reduced to a general BFGS method.

\textbf{\emph{Remark 2}.} In order to simplify the symbol and facilitate writing, we still record the updated symbol $\mu_{k+1}$ as $\mu$.

In the process of algorithm implementation, the search direction \eqref{(2.21)} in  subspace $\mathcal{S}_k$  always converts to the full space $\mathbb{R}^n$ at each RQN iteration, i.e.,
\begin{equation}\label{(2.26)}
{d}_{k+1} = -P_k {g}_{k+1},
\end{equation}
where
\begin{equation}\label{(2.27)}
P_k =  Z_k \hat{B}^{-1}_{k+1}(\mu)Z^T_k
\end{equation}
and $\hat{B}_{k+1}(\mu)$ is given by \eqref{(2.22)}.

In Section 3, we will show that matrices $\hat{B}_{k+1}(\mu)$ and $P_k$ have some good properties in the  RQN iteration, which is critical for the convergence analysis.

\subsection{ An Effective Acceleration Technique }
In order to optimize the performance of the algorithm, Sun et al. \cite{Sun21} proposed an acceleration technique, which replaces \eqref{(1.2)} with the following new iterative form
\begin{equation}\label{(2.28)}
x_{k+1} = x_{k} + \bar{\eta}_{k}\alpha_{k}d_{k},
\end{equation}
where $\bar{\eta}_{k} \geq 0 $ is an acceleration parameter obtained from an interpolation function. In view of the numerical effect of the acceleration technique, our algorithm also takes it into account. Similar to reference \cite{Sun21}, we minimize the following interpolation function to get the acceleration parameter $\bar{\eta}_{k}$:
\begin{equation}\label{(2.29)}
\bar{\eta}_{k} = \mathrm{arg} \ \mathrm{min} \ \ q( \varphi_{k}(\bar{\eta}) ),
\end{equation}
where $\bar{\eta} \geq 0, \  \varphi_{k}(\bar{\eta})=f(x_{k}+\bar{\eta} \alpha_{k}d_{k} ),$  and $ q( \varphi_{k}(\bar{\eta}) )$ represents the interpolation function defined by $\varphi_{k}(\bar{\eta})$.
In the paper, we consider minimizing the quadratic interpolation function \cite{Nocedal99} $ q( \varphi_{k}(0), \varphi^{\prime}_{k}(0), \varphi^{\prime}_{k}(1) ),$ then,
\begin{equation}\label{(2.30)}
\bar{\eta}_{k} = \mathrm{arg} \ \mathrm{min} \ \ q( \varphi_{k}(0), \varphi^{\prime}_{k}(0), \varphi^{\prime}_{k}(1)  ),
\end{equation}
By minimizing \eqref{(2.30)} we have
\begin{equation}\label{(2.31)}
\bar{\eta}_{k} = - \frac{\bar{a}_{k}}{\bar{b}_{k}}, \ \bar{b}_{k} \geq \bar{\epsilon},
\end{equation}
where  $\bar{a}_{k}= \alpha_{k}g^{T}_{k}d_{k},$ $\bar{b}_{k}=\alpha_{k}(g_{\bar{z}}-g_{k})^{T}d_{k},$ $g_{\bar{z}}=\nabla f(\bar{z}),$  $\bar{z}=x_{k}+\alpha_{k}d_{k}$ and $ \bar{\epsilon}>0$ is a small constant.

We propose the following acceleration criterion, which is simpler than the rule in reference { \cite{Sun21}}, that is
\begin{equation}\label{(2.32)}
\bar{b}_{k} \geq \bar{\epsilon},  \  \| s_{\bar{z}} \|^{2} \leq \bar{\tau},  \  \| g_{k} \|^{2} \leq \hat{\tau}, \ |\bar{t}_{k+1} |< {\bar{c}},  \ \mathrm{and} \ | s^{T}_{k}g_{\bar{z}}|\geq \mathrm{Max}(\varsigma,\bar{\varsigma}\cdot \bar{b}_{k})
\end{equation}
where $\bar{\epsilon}, \bar{\tau}, \hat{\tau},$ $\bar{c}, \varsigma $ and $\bar{\varsigma}$ are all small positive constants, $ \bar{b}_{k}=\alpha_{k}(g_{\bar{z}}-g_{k})^{T}d_{k},$ $s_{\bar{z}}= \bar{z} - x_{k}$, $\bar{z}=x_{k}+\alpha_{k}d_{k},$ $|\bar{t}_{k+1} |=| \frac{2(f_{k}-f_{\bar{z}}+g^{T}_{\bar{z}}s_{\bar{z}})}{s^{T}_{\bar{z}}g_{\bar{z}}} - 1 |$, $f_{\bar{z}}=f(\bar{z})$ and $g_{\bar{z}}=\nabla f(\bar{z}).$ When the condition \eqref{(2.32)} holds, we accelerate the algorithm and update the relevant variables. In addition, one of the necessary conditions for successful acceleration is that the trial iteration point must satisfy the line search condition. Therefore, if the algorithm accelerates successfully, update the iteration point $x_{k+1}$ by using \eqref{(2.28)}. Otherwise the algorithm acceleration fails and returns to the original algorithm, at which point $\bar{\eta}_k =1,$ update the iteration point $x_{k+1}$  with \eqref{(1.2)}.

In reference \cite{Sun21}, the acceleration criterion is divided into three cases, which seems to be more complex, while our acceleration criterion has only one case and the form is simpler.

\subsection{ Choices of the Initial Stepsize and the Generalized  Nonmonotone Wolfe Line Search }
It is well known that the design of the search direction and the conditions of the line search are two critical factors which affect the efficiency of the line search algorithm. In this subsection, we will develop an improved nonmonotone Wolfe line search which can be regarded as an extension of the Zhang-Hager's \cite{Zhang04} nonmonotone line search. In addition, an improved initial step selection strategy is designed.

For the sake of convenience, we express the one-dimensional line search function as
\begin{align*}
{\phi _k}(\alpha ) = f({x_k} + \alpha {d_k}),\alpha  \ge 0.
\end{align*}
The choice of the initial stepsize $\alpha^0_k$ is of great importance for a line search in an optimization method.  For the Newton-like methods, choosing the initial step $\alpha^0_k = 1$ is important to speed up convergence. For the conjugate gradient methods, it is essential to use information from the current iteration of the problem to make initial guesses \cite{Nocedal99}. In the conjugate gradient method, there have been various ways to choose the initial stepsize, for example, see \cite{Dai13,Hager05,Hager2013,Nocedal99}.   However, it did not have an agreement on which is the best.
In particular, Hager and Zhang \cite{Hager2013} select the initial step in CG\_DESCENT as below:
\begin{equation}\label{(2.33)}
\alpha^0_k = \left\{
\begin{array}{ll}
  \arg \min \bar{q}\left( {{\phi _k}\left( 0 \right),{{\phi '}_k}\left( 0 \right),{\phi _k}\left( \bar{\tau}_1 \alpha_{k-1} \right)} \right), & \; \mathrm{if} \; {\phi _k}\left( \bar{\tau}_1 \alpha_{k-1} \right) \leq {\phi _k}\left( 0 \right), \\
  \bar{\tau}_2 \alpha_{k-1},  & \mathrm{otherwise},
\end{array}\right.
\end{equation}
where $\bar{q}\left( {{\phi _k}\left( 0 \right),{{\phi '}_k}\left( 0 \right),{\phi _k}\left( \tau_1 \alpha_{k-1} \right)} \right)$  represents the interpolation function given by the three values ${\phi _k}\left( 0 \right),$ ${{\phi '}_k}\left( 0 \right)$ and ${\phi _k}\left( \tau_1 \alpha_{k-1} \right),$  $\bar{\tau}_1$ and $ \bar{\tau}_2$ are positive parameters. In CGOPT, Dai and Kou \cite{Dai13} determined the initial stepsize in the following way:
\begin{equation}\label{(2.34)}
\alpha^0_k = \left\{
\begin{array}{ll}
  \alpha & \; \mathrm{if} \;  \left| {\phi _k}\left( \alpha \right) -{\phi _k}\left( 0 \right) \right| / \left( \tau_3 + {\phi _k}\left( 0 \right) \right) > \tau_4, \\
  \arg \min \bar{q}\left( {{\phi _k}\left( 0 \right),{{\phi '}_k}\left( 0 \right),{\phi _k}\left( \alpha \right)} \right), \;  & \mathrm{otherwise},
\end{array}\right.
\end{equation}
where $\alpha = \max \left\{ \tau_5 \alpha_{k-1}, -2\left| f_k -f_{k-1} \right|/ {g^T_k d_k} \right\}$, $\tau_3 >0,$ $\tau_4 >0$ and $\tau_5 >0$. Most recently, Liu and Liu \cite{Liu19} discussed the development a very effective initial stepsize selection strategy  for SMCG method  by combining the BB methods and the interpolation technique.

Based on the above research, we devise an improved strategy to obtain the initial stepsize. We first consider the initial stepsize for the search direction in the RQN iteration.

(i) Initial stepsize of the search direction \eqref{(2.26)} with $B_{k+1}(\mu)\neq I$.

Since the search direction $\hat{d}$ is a quasi-Newton direction in the subspace $\mathcal{S}_k$, then the initial stepsize  $\alpha^0_k =1$ may be a good choice. Therefore, the trial initial stepsize can be stated as
\begin{equation}\label{(2.35)}
\alpha _k^0 = \left\{ \begin{array}{l}
{{\hat \alpha }_k},\;\;\;\;\;{\rm{    if\;\;\left( \eqref{(2.2)} \; or \; \varpi \leq \tau_2 \right)\;\; holds\;\; and\;\; }}{{\bar \alpha }_k} > 0,\\
1,\;\;\;\;\;\;\;{\rm{       otherwise,}}
\end{array} \right.
\end{equation}
where
\begin{align*}
{{\hat \alpha }_k} = \min \{ \max\{ {{\bar \alpha }_k},{\alpha_{\min }}\} ,{\alpha_{\max }}\},  \;\; {{\bar \alpha }_k} = \min \bar{q}({\phi _k}(0),{\phi _k}^\prime (0),{\phi _k}(1)),
\end{align*}
\begin{align*}
\varpi = \left| {\phi _k}\left( 1\right) - {\phi _k}\left( 0 \right) \right| / \left( \tau_1 + {\phi _k}\left( 0 \right) \right), \;\tau_1 >0, \; \tau_2 >0 \; \; {\rm{and}}\;\;{\alpha_{\max }} > {\alpha_{\min }} > 0.
\end{align*}
Here, $\bar{q}\left( {{\phi _k}\left( 0 \right),{{\phi '}_k}\left( 0 \right),{\phi _k}\left( 1 \right)} \right)$ is a quadratic interpolation function for ${{\phi _k}\left( 0 \right),}$ ${{{\phi '}_k}\left( 0 \right),}$ and ${{\phi _k}\left( 1 \right),}$ and ${\alpha_{\max }}$ and ${\alpha _{\min }}$ represent two positive constants.

(ii) Initial stepsize of the search direction \eqref{(2.26)} with $B_{k+1}(\mu) = I$.
\begin{equation}\label{(2.36)}
\alpha _k^0 = \left\{ \begin{array}{l}
{{\hat \alpha }_k},\;\;\;\;\;{\rm{    if\;\;\left( \eqref{(2.2)} \; or \; \varpi \leq \tau_2 \right)\;\; holds\;\; and\;\; }}{{\bar \alpha }_k} > 0,\\
\bar{\bar{\alpha}}_k,\;\;\;\;\;\;{\rm{otherwise,}}
\end{array} \right.
\end{equation}
where
\begin{equation}\label{(2.37)}
{\bar {\bar \alpha} _k} = \left\{ \begin{array}{l}
\max \{ \min \{ \alpha _k^{B{B_2}},{\alpha _{\max }}\} ,{\alpha _{\min }}\} ,\;\;{\rm{ if}}\;\;g_k^T{s_{k - 1}} > 0,\\
\max \{ \min \{ \alpha _k^{B{B_1}},{\alpha _{\max }}\} ,{\alpha _{\min }}\} ,\;\;{\rm{ if}}\;\;g_k^T{s_{k - 1}} \le 0,
\end{array} \right.
\end{equation}

For the initial stepsize of the search direction in the SMCG iteration. If the search direction $d_k$ is calculated by \eqref{(2.12)} with $d_k \neq -g_k$, the initial stepsize is chosen in the same way as the RQN iteration, which is determined by \eqref{(2.35)}. If the search direction $d_k$ is given by \eqref{(2.11)}, the initial stepsize is determined by
\begin{equation}\label{(2.38)}
\alpha _k^0 = \left\{ \begin{array}{l}
\min \{ \max\{ {\tilde{\tilde \alpha}  }_k,{\alpha _{\min }}\} ,{\alpha _{\max }}\} ,{\rm{\;\;if\;\;\eqref{(2.2)} \;\;holds,\;\;}} {\left\| {{g_k}} \right\|^2} \le 1, \;\; {d_{k - 1}} \ne  - {g_{k - 1}} {\rm{ \;\;and\;\; }}{\tilde{\tilde \alpha}  }_k > 0,\\
{{\bar {\bar \alpha} }_k},{\rm{ \quad \qquad\qquad\qquad\qquad\qquad otherwise}}{\rm{,}}
\end{array} \right.
\end{equation}
where ${{\bar {\bar \alpha} }_k}$ is determined by \eqref{(2.37)} and ${{\tilde{\tilde \alpha}  }_k} = \min q({\phi _k}(0),{\phi _k}^\prime (0),{\phi _k}({\bar {\bar \alpha} _k})).$

Next, we introduce a generalized line search condition, which can be regarded as a development of the Zhang-Hager's nonmonotone line search. We recall the nonmonotone line search introduced by Zhang and Hager \cite{Zhang04}
\begin{equation}\label{(2.39)}
f(x_{k} +  \alpha_k d_{k}) \le C_k + \delta \alpha_{k}  g_k^{ T}{d_k},
\end{equation}
where
\begin{equation}\label{(2.40)}
{C_{k + 1}} = \frac{{{\eta _k}{Q_k}{C_k} + f{_{k + 1}}}}{{{Q_{k + 1}}}}, \;\;  {Q_{k + 1}} = {\eta _k}{Q_k} + 1,
\end{equation}
$0 < \delta  < 1,$ and ${\eta _k} \in [0,1]$. From \eqref{(2.40)}, it is easy to see that $C_{k+1}$ is a convex combination of $f_{k+1}$ and $C_k.$ If $C_0 =f(x_0)$, it is thus clear that $C_k$ can be regard as a convex combination of the function values $f(x_0), f(x_1), \cdots, f(x_k)$. It means that $C_k$ can employ information about the known function values from the previous iteration. The Zhang-Hager's nonmonotone line search  \eqref{(2.39)}  is reduced to the standard Armijo line search condition when  $\eta_k=0$ for each $k$.

As it was reported in \cite{Zhang04}, the nonmonotone line search proposed by Zhang and Hager plays a crucial role in generating an appropriate stepsize compared to the monotone line search method. Based on \eqref{(2.39)} and \eqref{(2.40)}, Huang et al. \cite{Huang2015} presented a very effective nonmonotone line search technique, which can be regard as an extension of Zhang-Hager's nonmonotone line search, that is
\begin{equation}\label{(2.41)}
C_{k+1} = \frac{{{\eta _k}{Q_k}{C_k} + f_{k + 1}}}{{{Q_{k + 1}}}}  \le  C_k + \delta_{k} \alpha_{k}  g_k^{ T}{d_k},
\end{equation}
where $\eta_k \in [\eta_{\min}, \eta_{\max}]$, $\delta_{\max}<1,$ $0<\delta_{\min}<(1-\eta_{\max})\delta_{\max},$ $\delta_{\min}\leq \delta_{k} \leq \frac{\delta_{\max}}{Q_{k+1}}$ and $Q_{k+1}$ is computed by \eqref{(2.40)}.

Inspired by the previous discussion, we will study a generalized nonmonotone Wolfe line search technique based on \eqref{(2.40)} and \eqref{(2.41)}. Considering the acceleration technique, the generalized nonmonotone Wolfe line search conditions are as follows:
\begin{equation}\label{(2.42)}
C_{k+1}  \le  C_k + \delta_{k} \bar{\eta}_k \alpha_{k}  g_k^{ T}{d_k},
\end{equation}
\begin{equation}\label{(2.43)}
g_{k + 1}^{T}{d_k} \ge \sigma g_k^{T}{d_k},
\end{equation}
where $0<\delta_{\min}<\delta_k<\delta_{\max}<1,$ $ \sigma\in(0,1),$ ${Q_0} = 1,$ ${C_0} = {f_0},$  $\bar{\eta}_{k}$ is an acceleration parameter  determined by \eqref{(2.31)}, ${C_k}$ and ${Q_k}$ are updated as follows
\begin{equation}\label{(2.44)}
 {C_{k + 1}} = \frac{{{\eta _k}{Q_k}{C_k} + f({x_{k + 1}})}}{{{Q_{k + 1}}}},\; {Q_{k + 1}} = {\eta _k}{Q_k} + 1,\; f({x_{k + 1}})=f(x_k + \bar{\eta}_k \alpha_k d_k),
\end{equation}
where ${{\eta} _k} \in [0,1]$. Specially,
\begin{equation}\label{(2.45)}
 {Q_1} = 2.0, \;\;  {C_1} = \min \{ {C_0},{f_1} + 1.0\},
\end{equation}
when $k \ge 1,$ ${C_{k + 1}}$ and ${Q_{k + 1}}$ are updated by  \eqref{(2.44)}, and ${{\eta} _k}$ is given as
\begin{equation}\label{(2.46)}
{\eta} _k = \left\{
\begin{array}{ll}
  1, & \; \mathrm{if} \; C_k -f_{k+1} > 0.95|C_k| \; \mathrm{and} \; k>100, \\
  0.9, & \; \mathrm{otherwise}.
\end{array}
\right.
\end{equation}
Here $\eta_k$ is a parameter that controls the degree of non-monotonicity, referred to \cite{Liu18b}.

Furthermore, we demonstrate that the generalized nonmonotone Wolfe line search is an extension of the Zhang-Hager's nonmonotone Wolfe line search method. It follows from \eqref{(2.42)} that we get
\begin{equation}\label{(2.47)}
f(x_{k} + \bar{\eta}_k \alpha_k d_{k}) \le (Q_{k+1}-\eta_k Q_k) C_k + Q_{k+1} \delta_k \bar{\eta}_k \alpha_{k}  g_k^{ T}{d_k}.
\end{equation}
Since $Q_{k+1}-\eta_k Q_k=1$, \eqref{(2.42)} is equivalent to
\begin{equation}\label{(2.48)}
f(x_{k} +  \bar{\eta}_k\alpha_k d_{k}) \le C_k + Q_{k+1} \delta_k \bar{\eta}_k \alpha_{k}  g_k^{ T}{d_k},
\end{equation}
It is easy to see that if $\delta_k = \frac{\delta}{Q_{k+1}}$,  nonmonotone line search condition \eqref{(2.48)}  reduces to the
Zhang-Hager's nonmonotone Wolfe line search condition \eqref{(2.39)}. This means that the Zhang-Hager's nonmonotone Wolfe line search condition in \cite{Zhang04} can be considered as a particular version of \eqref{(2.42)}.

\subsection{ A Regularized Limited Memory Subspace Minimization Conjugate Gradient Algorithm(RL\_SMCG) }
In this subsection, we describe the regularized limited memory subspace minimization conjugate gradient algorithm in detail. As mentioned above, the regularized limited memory subspace minimization conjugate gradient algorithm is made of two kinds of iterations. The ``state" in Algorithm 1 represents for the type of iteration, i.e., state= ``SMCG" means that SMCG iteration will be carried out, and state= ``RQN" means that RQN iteration will be performed.

\vspace{0.3cm}
\begin{breakablealgorithm}
\begin{spacing}{1.4}
	\caption{RL\_SMCG}
\noindent  \textbf{Step 0.} Chosen ${x_0} \in {\mathbb{R}^n},$ $\varepsilon  > 0,$  ${\tilde{\eta}}_{0},$ ${\tilde{\eta}}_{1},$ $\upsilon,$ $m,$ $\xi_{1},$ $\xi_{2},$ $\xi_{3},$ $\xi_{4},$ $\xi_{5},$ $\sigma_1,$ $\sigma_2,$ $\sigma_3,$ $\mu_{\min},$ $\mu_{\max},$ $\tau,$ $\bar{\tau},$ $\bar{c},$ $\varsigma,$ $\bar{\varsigma},$ $\bar{\epsilon},$ $\tau_1,$ $\tau_2,$ $\delta_k,$  $\sigma,$ IterRestart $:=0$, IterQuad $:=0$ and MinQuad. Set state = ``SMCG" and $k:=0$.\\
\textbf{Step 1.} If ${\left\| {{g_k}} \right\|_\infty } \le \varepsilon $, stop.\\
\textbf{Step 2.} Compute the search direction.\\
\hspace*{1.2cm} If (state = ``SMCG"), then\\
\hspace*{1.8cm} If $k=0,$ then $d_0 = -g_0.$\\
\hspace*{1.8cm} elseif (IterQuad $=$ MinQuad and IterQuad $\neq$ IterRestart), set\\
\hspace*{2.5cm} $d_k = -g_k,$  IterQuad $=0,$ and IterRestart $=0.$ \\
\hspace*{1.8cm} else \\
\hspace*{2.5cm} Determine the search direction $d_k$ by \eqref{(2.12)}.\\
\hspace*{1.8cm} end\\
\hspace*{1.2cm} elseif (state = ``RQN"), then \\
\hspace*{1.8cm} Compute $P_k$ by \eqref{(2.27)}, and compute the search direction $d_k$ by \eqref{(2.26)}.\\
\hspace*{1.2cm} end\\
\textbf{Step 3.} Determine the corresponding initial step size ${\alpha^{0}_k}$ from \eqref{(2.35)}, \eqref{(2.36)} and \eqref{(2.38)} according to the different iteration directions in the Step 2.\\
\textbf{Step 4.} Determine a stepsize $\alpha_k$ satisfying the generalized nonmonotone Wolfe line search \eqref{(2.42)} and \eqref{(2.43)} with initial stepsize ${\alpha^{0}_k}$.\\
\textbf{Step 5.}Compute the trial iteration $\bar{z}=x_{k}+\alpha_{k}d_{k}$ and $g_{\bar{z}}=\nabla f(\bar{z})$. If ${\left\| {{g_{\bar{z}}}} \right\|_\infty } \le \varepsilon$, then stop; otherwise, go to Step 6.\\
\textbf{Step 6.} Acceleration procedure.\\
\hspace*{1.2cm}If the condition \eqref{(2.32)} holds, then go to 6.1.\\
\hspace*{1.8cm}\textbf{6.1.} \ Compute $\bar{a}_{k}= \alpha_{k}g^{T}_{k}d_{k} $, $ \bar{b}_{k}=\alpha_{k}(g_{\bar{z}}-g_{k})^{T}d_{k}$ and  $\bar{\eta}_{k}$ by \eqref{(2.31)}.\\
\hspace*{1.8cm}\textbf{6.2.} \ Update the iteration point as ${x_{k + 1}} = {x_k} + \bar{\eta}_{k}{\alpha _k}{d_k}$ and compute $f_{k+1}$ and $g_{k+1}.$\\
\hspace*{1.8cm}\textbf{6.3.} \ If $f_{k+1}$ satisfies \eqref{(2.42)} and  $g_{k+1}$ satisfies \eqref{(2.43)}, go to Steps 8. Otherwise, go to Steps 7. \\
\hspace*{1.2cm} else\\
\hspace*{1.8cm} go to Steps 7.\\
\hspace*{1.2cm} end\\
\textbf{Step 7.} Update the variable as ${x_{k + 1}} = {x_k} + {\alpha _k}{d_k}$. Compute $f_{k+1}$ and $g_{k+1}$. \\
\textbf{Step 8.} Update restart conditions. \\
\textbf{Step 9.} Update $ Q_{k+1}$ and $C_{k+1}$ with \eqref{(2.44)}.\\
\textbf{Step 10.} Update iteration type.\\
\hspace*{1.2cm} If (state = ``SMCG"), then\\
\hspace*{1.8cm} If  \eqref{(2.16)} holds, then state = ``RQN".\\
\hspace*{1.2cm} elseif (state = ``RQN"), then \\
\hspace*{1.8cm} If  \eqref{(2.17)} holds, then state = ``SMCG".\\
\hspace*{1.2cm} end\\
\textbf{Step 11.} Set $k:=k+1$ and go to Step 1.
    \end{spacing}
\end{breakablealgorithm}
\vspace{0.5cm}

\textbf{\emph{Remark 3}.} Notably, when the lost orthogonality is corrected, our algorithm  terminates the RQN iteration and immediately calls the SMCG iteration. However, the limited memory CG method \cite{Hager2013} first carries out  the complex preprocessing CG iteration after the orthogonality is improved. This means that algorithm RL\_SMCG is more simple compared to the limited memory CG method \cite{Hager2013}.

\section{Convergence Analysis}\label{sec3}

In the section, we establish the global convergence of the algorithm RL\_SMCG under the following assumptions and properties.

Define $\mathcal{N} $ to be an open neighborhood of the level set $L\left( {{x_0}} \right) = \left\{ {x \in {R^n}:f\left( x \right) \le f\left( {{x_0}} \right)} \right\},$ where ${{x_0}}$ is an initial point.

\begin{assumption}\label{Ass1}
(i) The objective function $f$ is continuously differentiable in $\mathcal{N}$ and the level set is bounded from below.
(ii) The gradient $g$ of the objective function is Lipschitz continuous in $\mathcal{N},$ i.e., there exists a constant $L > 0$ such that $\left\| {g(x) - g(y)} \right\| \le L\left\| {x - y} \right\|, \forall x, y \in \mathcal{N}.$
\end{assumption}

Under these assumptions, we have the following several properties.
\begin{lemma}\label{Lem1}
Suppose that Assumption \ref{Ass1} holds. Then, for $\hat{B}_{k+1}(\mu )$ in \eqref{(2.22)}, there exist three constants $\hat{\xi}_{1}>0, \hat{\xi}_{2}>0 $ and $\hat{\xi}_{3}>0$ such that
\[
\lambda_{\max}\left(\hat{B}_{k+1} (\mu )\right)\leq \hat{\xi}_{1}, \ \lambda_{\max}\left(\hat{B}^{-1}_{k+1} (\mu )\right)\leq \hat{\xi}_{2}, \ \left\|\hat{B}^{-1}_{k+1} (\mu )\right\| \leq \hat{\xi}_{3}.
\]
\end{lemma}
\begin{proof}
We know that $Z_{k}$ is a normal orthogonal basis of $\mathcal{S}_{k}$ and the dimension $m < +\infty,$ hence we have $\xi_0>0$ such that $\|Z_{k}\| \leq \xi_0.$ According to  \eqref{(2.22)} and the property of the matrix norm in finite dimensional spaces, we can get that $\lambda_{\max}\left(\hat{B}_k (\mu)\right)=1$ or
\begin{align}\label{(3.1)}
\lambda_{\max}\left(\hat{B}_{k+1} (\mu)\right) & \leq \lambda_{\max}\left(\hat{B}_k (\mu)\right) + \lambda_{\max}\left(-\frac{\hat{B}_k (\mu)\hat{s}_k \hat{s}^T_k \hat{B}_k (\mu)}{\hat{s}^T_k\hat{B}_k (\mu)\hat{s}_k}  \right) + \lambda_{\max}\left(\frac{\hat{y}_k(\mu) \hat{y}^T_k(\mu)}{\hat{s}^T_k\hat{y}_k(\mu)} \right) \\ \nonumber
& \leq \lambda_{\max}\left(\hat{B}_k (\mu)\right) + \frac{\hat{y}^T_k(\mu) \hat{y}_k(\mu)}{\hat{s}^T_k\hat{y}_k(\mu)}.
\end{align}
Further, by $\hat{y}_k(\mu) = \hat{y}_k + \mu \hat{s}_k,$ $\mu>0$, we get
\begin{align*}
\frac{\hat{y}^T_k(\mu) \hat{y}_k(\mu)}{\hat{s}^T_k\hat{y}_k(\mu)} &= \frac{\|\hat{y}_k\|^{2} +\mu^2_k \|\hat{s}_k\|^2 +  2\mu \hat{s}^T_k \hat{y}_k }{\hat{s}^T_k \hat{y}_k +\mu \|\hat{s}_k\|^2}  \\ \nonumber
& = \frac{\|\hat{y}_k\|^{2}+ \mu \hat{s}^T_k \hat{y}_k}{\hat{s}^T_k \hat{y}_k +\mu \|\hat{s}_k\|^2} +\frac{\mu \hat{s}^T_k \hat{y}_k +\mu^2_k \|\hat{s}_k\|^2}{\hat{s}^T_k \hat{y}_k +\mu \|\hat{s}_k\|^2} \\ \nonumber
& \leq \frac{\|\hat{y}_k\|^{2}+ \mu \hat{s}^T_k \hat{y}_k}{\hat{s}^T_k \hat{y}_k } + \mu \\ \nonumber
& \leq \frac{L^2 \xi^2_0 \|\hat{s}_k\|^2}{\hat{s}^T_k \hat{y}_k } +2\mu \\ \nonumber
& \leq \frac{L^2 \xi^2_0 }{\upsilon}  +2\mu_{\max}.
\end{align*}
The fourth inequality above is obtained from $\hat{y}_k = Z^T_k y_k, \|Z_k\| \leq \xi_0$ and Assumption \ref{Ass1} (ii).
Because $\hat{B}_k(\mu)$ will be set to $\hat{I}$ after a maximum of $l$ updates, combining with \eqref{(3.1)} easy to get $\lambda_{\max}\left(\hat{B}_{k+1} (\mu)\right)  \leq 1+ \frac{l L^2 \xi^2_0 }{\upsilon} +2l\mu_{\max} \triangleq \hat{\xi}_1. $

Let $\hat{P}_k(\mu) = \hat{B}^{-1}_{k+1}(\mu).$ According to  \eqref{(2.22)} and some simple matrix operations, we have that $\hat{P}_k(\mu) = \hat{I}$ or
\begin{equation}\label{(3.2)}
\hat{P}_k(\mu) = \left( \hat{I} - \frac{\hat{y}_k(\mu) \hat{s}^T_k}{\hat{s}^T_k\hat{y}_k(\mu)}\right)^T \hat{P}_{k-1}(\mu)\left( \hat{I} - \frac{\hat{y}_k(\mu)\hat{s}^T_k}{\hat{s}^T_k\hat{y}_k(\mu)}\right) +\frac{\hat{s}_k\hat{s}^T_k}{\hat{s}^T_k\hat{y}_k(\mu)}.
\end{equation}
It is not difficult to that $\lambda_{\max}\left( \left( \hat{I} - \frac{\hat{y}_k(\mu)\hat{s}^T_k}{\hat{s}^T_k\hat{y}_k(\mu)}\right)^T \left( \hat{I} - \frac{\hat{y}_k(\mu)\hat{s}^T_k}{\hat{s}^T_k\hat{y}_k(\mu)}\right) \right) =\frac{\|\hat{y}_k(\mu)\|^2 \|\hat{s}_k\|^2}{\left(\hat{s}^T_k\hat{y}_k(\mu)\right)^2 } $. For any $\hat{z}\neq 0 \in \mathbb{R}^m$ and $\hat{P}_k(\mu)$ in \eqref{(3.2)}, we have
\begin{align*}
\hat{z}^T \hat{P}_k(\mu)\hat{z} &= \hat{z}^T \left( \hat{I} - \frac{\hat{y}_k(\mu) \hat{s}^T_k}{\hat{s}^T_k\hat{y}_k(\mu)}\right)^T \hat{P}_{k-1}(\mu)\left( \hat{I} - \frac{\hat{y}_k(\mu)\hat{s}^T_k}{\hat{s}^T_k\hat{y}_k(\mu)}\right)\hat{z} + \frac{\left(\hat{s}^T_k\hat{z}\right)^2}{\hat{s}^T_k\hat{y}_k(\mu)} \\ \nonumber
& \leq \lambda_{\max} \left(\hat{P}_{k-1}(\mu)\right) \hat{z}^T \left( \hat{I} - \frac{\hat{y}_k(\mu)\hat{s}^T_k}{\hat{s}^T_k\hat{y}_k(\mu)}\right)^T \left( \hat{I} - \frac{\hat{y}_k(\mu)\hat{s}^T_k}{\hat{s}^T_k\hat{y}_k(\mu)}\right) \hat{z} + \frac{\left(\hat{s}^T_k\hat{z}\right)^2}{\hat{s}^T_k\hat{y}_k(\mu)} \\ \nonumber
& \leq  \lambda_{\max} \left(\hat{P}_{k-1}(\mu)\right) \lambda_{\max}\left( \left( \hat{I} - \frac{\hat{y}_k(\mu)\hat{s}^T_k}{\hat{s}^T_k\hat{y}_k(\mu)}\right)^T \left( \hat{I} - \frac{\hat{y}_k(\mu)\hat{s}^T_k}{\hat{s}^T_k\hat{y}_k(\mu)}\right) \right) \|\hat{z}\|^2 + \frac{\left(\hat{s}^T_k\hat{z}\right)^2}{\hat{s}^T_k\hat{y}_k(\mu)} \\ \nonumber
& \leq  \lambda_{\max} \left(\hat{P}_{k-1}(\mu)\right) \frac{\|\hat{y}_k(\mu)\|^2 \|\hat{s}_k\|^2}{\left(\hat{s}^T_k\hat{y}_k(\mu)\right)^2 } \|\hat{z}\|^2 + \frac{\|\hat{s}_k\|^2}{\hat{s}^T_k\hat{y}_k(\mu)}\|\hat{z}\|^2.
\end{align*}
The above inequality is divided by $\|\hat{z}\|^2$, and the resulting inequality is maximized, then we have
\begin{align*}
\lambda_{\max} \left(\hat{P}_{k }(\mu)\right) & \leq \lambda_{\max} \left(\hat{P}_{k-1}(\mu)\right) \frac{\|\hat{y}_k(\mu)\|^2 \|\hat{s}_k\|^2}{\left(\hat{s}^T_k\hat{y}_k(\mu)\right)^2 } + \frac{\|\hat{s}_k\|^2}{\hat{s}^T_k\hat{y}_k(\mu)}   \\ \nonumber
& \leq  \lambda_{\max} \left(\hat{P}_{k-1}(\mu)\right) \left( \frac{\|\hat{y}_k(\mu)\|^2}{ \hat{s}^T_k\hat{y}_k(\mu) \frac{  \|\hat{s}_k\|^2}{ \hat{s}^T_k\hat{y}_k(\mu)  } } \right) + \frac{\|\hat{s}_k\|^2}{\hat{s}^T_k\hat{y}_k } \\ \nonumber
& \leq \lambda_{\max} \left(\hat{P}_{k-1}(\mu)\right) \left( \frac{L^2 \xi^2_0}{\upsilon} +2\mu_{\max}\right) \frac{\|\hat{s}_k\|^2}{\hat{s}^T_k\hat{y}_k }  + \frac{\|\hat{s}_k\|^2}{\hat{s}^T_k\hat{y}_k }  \\ \nonumber
& \leq \left(\frac{L^2 \xi^2_0}{\upsilon^2}+\frac{2\mu_{\max}}{\upsilon} \right) \lambda_{\max} \left(\hat{P}_{k-1}(\mu)\right) + \frac{1}{\upsilon}.
\end{align*}
The third inequality above is obtained from $\hat{y}_k = Z^T_k y_k, \|Z_k\| \leq \xi_0$ and Assumption \ref{Ass1} (ii).
Because $\hat{P}_k(\mu)$ will be set to $\hat{I}$ after a maximum of $l$ updates, it is easy to know that there exists a constant $\hat{\xi}_2 >0$ such that $\lambda_{\max} \left(\hat{B}^{-1}_{k+1}(\mu)\right) = \lambda_{\max} \left(\hat{P}_{k }(\mu)\right) \leq \hat{\xi}_2.$

Since $\hat{B}^{-1}_{k+1}(\mu)$ is a  positive definite and symmetric matrix, we have $\left\|\hat{B}^{-1}_{k+1}(\mu)\right\|_2 = \lambda_{\max} \left(\hat{B}^{-1}_{k+1}(\mu)\right) \leq \hat{\xi}_2.$  As a result, using the equivalence property of matrix norm in a finite dimensional space, it follows that there exists a constant $\hat{\xi}_3 >0$ such that $\left\|\hat{B}^{-1}_{k+1}(\mu)\right\| \leq \hat{\xi}_3.$ The proof is completed.
\qed
\end{proof}

\begin{lemma}\label{Lem2}
Suppose that Assumption \ref{Ass1} holds. Then, for $P_k$ in \eqref{(2.27)}, there exist three constants $\gamma_0 >0, \gamma_1 >0$ and $\gamma_2 >0$ such that
\begin{equation}\label{(3.3)}
\left\|P_k\right\|  \leq \gamma_0, \ g^T_{k+1}P_k g_{k+1} \geq \gamma_1 \left\|g_{k+1}\right\|^2, \ d^T_k P^{-1}_k d_k \geq \gamma_2 \left\|d_k\right\|^2,
\end{equation}
where $P^{-1}_k$ denotes the pseudoinverse of $P_k$.
\end{lemma}
\begin{proof}
By \eqref{(2.17)}, \eqref{(2.27)} and Lemma \ref{Lem1}, we obtain that
\begin{align*}
\left\|P_k\right\| =  \left\|Z_k \hat{B}^{-1}_{k+1}(\mu)Z^T_k \right\|=\left\| \hat{B}^{-1}_{k+1}(\mu) \right\| \leq \hat{\xi}_3 \triangleq \gamma_0,
\end{align*}
\begin{align*}
g^T_{k+1}P_k g_{k+1} & = g^T_{k+1}Z_k \hat{B}^{-1}_{k+1}(\mu)Z^T_k g_{k+1} \\ \nonumber
& = \hat{g}^T_{k+1} \hat{B}^{-1}_{k+1}(\mu)  \hat{g}_{k+1} \\ \nonumber
& \geq \lambda_{\min}\left(\hat{B}^{-1}_{k+1}(\mu)\right) \left\| \hat{g}_{k+1}\right\|^2  \\ \nonumber
& \geq \frac{1}{\hat{\xi}_1}\left(1-\tilde{\eta}^2_1\right)\left\|  {g}_{k+1}\right\|^2 \triangleq \gamma_1 \left\|g_{k+1}\right\|^2,
\end{align*}
\begin{align*}
d^T_k P^{-1}_k d_k = d^T_k Z_k \hat{B}^{-1}_{k+1}(\mu)Z^T_k d_k = \hat{d}^T_k \hat{B}^{-1}_{k+1}(\mu) \hat{d}_k \geq \frac{1}{\hat{\xi}_2} \left\|\hat{d}_k\right\|^2 =\frac{1}{\hat{\xi}_2} \left\|d_k\right\|^2 \triangleq \gamma_2 \left\|d_k\right\|^2.
\end{align*}
Therefore, we can get the conclusions. The proof is completed.
\qed
\end{proof}

Subsequently, we provide some properties of the search directions produced by the algorithm RL\_SMCG, which are crucial for the following convergence analysis.

\begin{lemma}\label{Lem3}
Suppose that Assumption \ref{Ass1} holds. Then, there exists a constant $c_1>0$ such that
the search directions  \eqref{(2.12)} and \eqref{(2.26)} are calculated by  algorithm  RL\_SMCG satisfy the sufficient descent condition:
\begin{equation}\label{(3.4)}
g_{k}^T{d_{k}} \leq  - {\bar{c}_1}{\left\| {{g_{k}}} \right\|^2}.
\end{equation}
\end{lemma}
\begin{proof}
We divide the proof into the following two cases.

(i) SMCG iteration. Similar to the proof of Lemma 4.1 of \cite{Zhao21}, it is easy to have
\[
g_{k}^T{d_{k}} \leq  - { {c}_1}{\left\| {{g_{k}}} \right\|^2},
\]
where $ c_{1}=\min \left\{ {\frac{1}{2},1 - {\bar{\xi}_3},\frac{2}{{3{\bar{\xi}_2}}},\frac{1}{{3{\bar{\xi}_2}}},\frac{2}{{5{\bar{\xi}_2}}}} \right\} $ .

(ii) RQN iteration. According to Lemma \ref{Lem2}, we have
\[
g_{k}^T{d_{k}} = -g_{k}^T P_{k-1}g_{k} \leq -\gamma_1 \left\|g_{k} \right\|^2.
\]

By setting $\bar{c}_1 = \min \left\{ c_{1}, \gamma_1 \right\}$, we can obtain \eqref{(3.4)}. The proof is completed.
\qed
\end{proof}

\begin{lemma}\label{Lem4}
Suppose that Assumption \ref{Ass1} holds. Then, there exists a constant $c_1>0$ such that
the search directions  \eqref{(2.12)} and \eqref{(2.26)} are calculated by  algorithm  RL\_SMCG satisfy
\begin{equation}\label{(3.5)}
\left\| d_k \right\|\leq  \bar{c}_2 \|g_{k}\|.
\end{equation}
\end{lemma}
\begin{proof}
We divide the proof into the following two cases.

(i) SMCG iteration. Referring to the proof procedure of Lemma 4.2 of \cite{Zhao21}, it is easy to get
\[
\left\| d_k \right\| \leq  {c}_2 \|g_{k}\|,
\]
where ${c_2} = \max  \left\{ {1,1 + \frac{L}{{{\bar{\xi}_1}}},\frac{{20}}{{{\bar{\xi}_1}}}} \right\} $.

(ii) RQN iteration. According to Lemma \ref{Lem2}, we obtain $\left\| d_k \right\| = \left\|- P_{k-1}g_{k}\right\| \leq \gamma_0 \left\|g_{k} \right\|$.

By setting $\bar{c}_2 = \min \left\{ c_{2}, \gamma_0 \right\}$, we can obtain \eqref{(3.5)}. The proof is completed.
\qed
\end{proof}

The following lemmas are very critical for the convergence analysis of algorithm  RL\_SMCG.

\begin{lemma}\label{Lem5}
Suppose that  Assumption \ref{Ass1} holds, and the sequence $\{ {x_k}\} $ is generated by the algorithm  RL\_SMCG.
Then,

If acceleration succeeds:
	\begin{equation}\label{(3.6)}
{\bar{\eta}_{k}\alpha _k} \ge \left( {\frac{{ 1 - \sigma }}{L}} \right)\frac{ \left| {g_k^T{d_k}} \right| }{{{{\left\| {{d_k}} \right\|}^2}}}.
	\end{equation}

If acceleration fails:
	\begin{equation}\label{(3.7)}
{\alpha _k} \ge \left( {\frac{{ 1 - \sigma }}{L}} \right)\frac{ \left| {g_k^T{d_k}} \right| }{{{{\left\| {{d_k}} \right\|}^2}}}.
	\end{equation}

Where $\sigma$ are given by  \eqref{(2.43)}.
\end{lemma}

\begin{proof}

We divide the proof into the following two cases.

\textbf{(i)} If acceleration succeeds:

From \eqref{(2.43)} and Assumptions \ref{Ass1} (ii), we obtain that
$$
(\sigma -1)g^{T}_{k}d_{k} \leq g(x_k + \bar{\eta}_k \alpha_k d_k)^{T}d_k - g^{T}_{k}d_{k} = (g(x_k + \bar{\eta}_k \alpha_k d_k) - g_{k})^{T}d_{k} \leq L\bar{\eta}_k \alpha_k \|d_k\|^{2},
$$
which yields
$$
{\bar{\eta}_k\alpha _k} \ge \left( {\frac{{\sigma - 1 }}{L}} \right)\frac{ {g_k^T{d_k}} }{{{{\left\| {{d_k}} \right\|}^2}}}.
$$
This means that \eqref{(3.6)} holds.

\textbf{(ii)} If acceleration fails:

 Let $\bar{\eta}_{k}=1,$ and the rest of the proof procedure is the same as before.
\qed
\end{proof}

\begin{lemma}\label{Lem6}
Suppose that Assumption \ref{Ass1} holds, and the sequence $\{ {x_k}\} $ is generated by the algorithm RL\_SMCG. Then, there holds that ${f_k} \le {C_k}$ for each $k$.
\end{lemma}

\begin{proof}
We divide the proof into the following two cases.

\textbf{(i)} If acceleration succeeds:

The new iterative update format is $x_{k+1}=x_{k}+\bar{\eta}_k \alpha_k d_k$, where $\bar{\eta}_{k} = - \frac{\bar{a}_{k}}{\bar{b}_{k}}$. Through  \eqref{(2.48)}, we have $f_{k+1}=f(x_{k}+\bar{\eta}_k \alpha_k d_k) \leq C_k + Q_{k+1}\delta_{k} \bar{\eta}_k \alpha_{k}g_k^{ T}{d_k}$. Combining \eqref{(2.44)}, $\delta_{k}>0$, lemma \ref{Lem5} and the sufficiently descent property of the direction $d_{k+1}$, we have $f_{k+1}<C_k$. The remaining proof process refers to Lemma 5.1 in \cite{Zhao21}, we can obtain $f_{k+1}\leq C_{k+1}$, hence ${f_k} \le {C_k}$ is established for each $k$.

\textbf{(ii)} If acceleration fails:

Let $\bar{\eta}_{k}=1,$ and the rest of the proof procedure is the same as before.
\qed
\end{proof}

\begin{theorem}\label{Th1}
Suppose that Assumption \ref{Ass1} holds, the  sequence $\{ {x_k}\} $ is generated by the algorithm RL\_SMCG. Then,
	\begin{equation}\label{(4.1)}
	\mathop {\lim}\limits_{k \to \infty } \left\| {{g_k}} \right\| = 0.
	\end{equation}
\end{theorem}

\begin{proof}
We divide the proof into the following two cases.

\textbf{(i)} If acceleration succeeds:

By  Assumptions \ref{Ass1}, lemmas \ref{Lem3} - \ref{Lem5} and  the generalized nonmonotone Wolfe line search conditions \eqref{(2.42)} and \eqref{(2.43)}, we get that
\begin{align}\label{(3.9)}
C_{k+1} & \leq C_k + \delta_{k} \bar{\eta}_k \alpha_{k}g_k^{T}{d_k} \\\nonumber
        & \leq C_k + \delta_{\min} \bar{\eta}_k \alpha_{k}g_k^{T}{d_k} \\\nonumber
        & \leq C_k + \delta_{\min} \frac{1-\sigma}{L}\frac{(g^T_k d_k)^2}{\|d_k\|^2} \\\nonumber
        & \leq C_k + \frac{\delta_{\min}(1-\sigma)\bar{c}_1^2}{L\bar{c}_2^2} \|g_k\|^2 \\\nonumber
        & = C_k + \beta\|g_k\|^2.
\end{align}
Where $\beta=\frac{\delta_{\min}(1-\sigma)\bar{c}_1^2}{L\bar{c}_2^2}.$ Combined with \eqref{(2.45)}, we have $C_1\leq C_0$ that means that $C_k$ is monotonically decreasing. According to  lemma \ref{Lem6}  and Assumption \ref{Ass1} (i), we know $C_k$ is bounded from below. Then
\begin{align*}
\sum\limits_{k = 0}^\infty  {{\beta }{{\left\| {{g_k}} \right\|}^2} < \infty } ,
\end{align*}
therefore,
\begin{align*}
\mathop {\lim }\limits_{k \to \infty } \left\| {g({x_k})} \right\| = 0.
\end{align*}

\textbf{(ii)} If acceleration fails:

Let $\bar{\eta}_{k}=1,$ and the rest of the proof procedure is the same as before.
\qed
\end{proof}

\section{Numerical Experiments}\label{sec4}
In this section, we compare the numerical performance of RL\_SMCG with ASMCG\_PR \cite{Sun21}, CG\_DESCENT(6.8) \cite{Hager2013} and CGOPT(2.0) \cite{Liu20} for the 145 test problems from  CUTEr library \cite{Gould03}. The codes of CG\_DESCENT(6.8) \cite{Hager2013} and CGOPT(2.0) \cite{Liu20} can be downloaded from \textcolor{blue}{http://users.clas.ufl.edu/hager/papers/Software} and \textcolor{blue}{https://web.xidian.edu.cn/xdliuhongwei/en/paper.html} or \textcolor{blue}{ http://lsec.cc.ac.cn/~dyh/software.html}, respectively.

In the numerical experiments, we set the parameters of RL\_SMCG as:
$\bar{\xi}_1 =10^{-10},$ $\bar{\xi}_2 =1.2 \times 10^{4},$ $\bar{\xi}_3 =5\times 10^{-5},$ $\bar{\xi}_4 =10^{-4},$ $\bar{\xi}_5 =0.08,$ $\tilde{\eta}_0 = 10^{-9},$ $\tilde{\eta}_1 =0.5,$ $\upsilon=5 \times 10^{-7},$ $m=\min\{n,11\},$ $\sigma_1 = 0.1,$ $\sigma_2 = 5,$ $\sigma_3 = 0.85,$ $\hat{\tau} =1,$ $\bar{\tau} =0.225,$ $\bar{c}=0.1,$ $\varsigma=5\times 10^{-5}(n\leq11),$ $\varsigma=5\times 10^{-6}(n>11),$ $\bar{\varsigma}=5\times 10^{-3},$ $\tau_1 =0.1,$ $\tau_2 =135,$ $\delta_k =0.0005$ and $\sigma=0.9999.$
CG\_DESCENT(6.8) and CGOPT(2.0) take the default parameters in their codes but the stopping conditions. Note that the number of memory $m$ for RL\_SMCG is $\min\{n,11\}$ while the number of memory for CG\_DESCENT(6.8) is 11.
All test methods in the experiment are terminated if ${\left\| {{g_k}} \right\|_\infty } \le {10^{ - 6}}$ is satisfied, and we set the number of iterations for all test algorithms to be no more than 200,000. In addition, all algorithms are running in Ubuntu 10.04 LTS.

We will show the performances of the test methods using the performance profiles introduced  by Dolan and Mor\'e \cite{Dolan02}. In the following Figs. \ref{fig.1}-\ref{fig.12}, ``${N_{iter}}$",``${N_f}$",``${N_g}$" and ``${T_{cpu}}$" represent the number of iterations, the number of function evaluations, the number of gradient evaluations and CPU time(s), respectively.

We divided the numerical experiments in three teams.
\begin{figure}[htp]
	\centering
	\begin{minipage}[t]{0.46 \linewidth}
       \centering
		\includegraphics[scale=0.5]{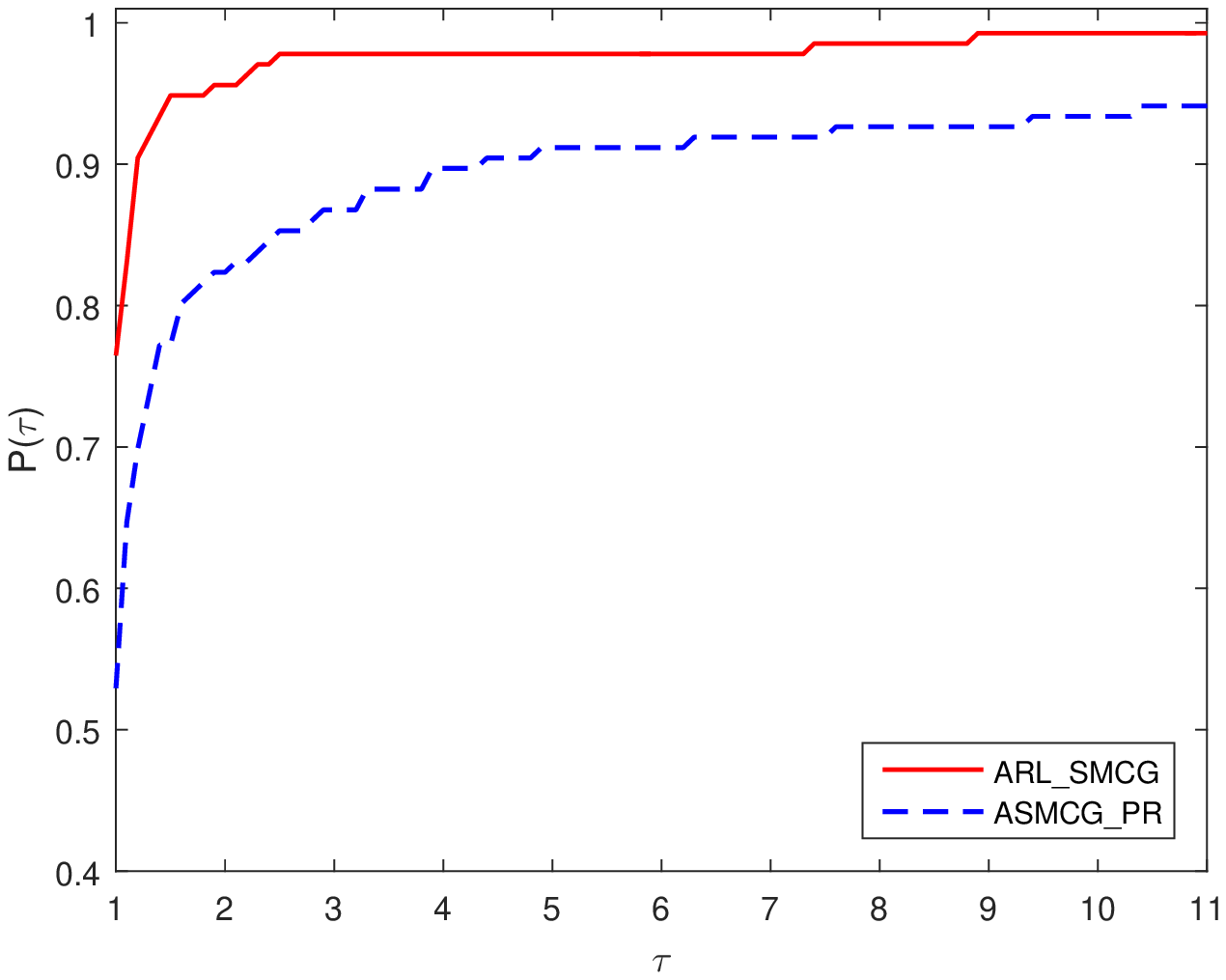}
%		\caption{Performance profile based on ${N_{iter}}$(CUTEr).}\label{fig.1}
\caption{${N_{iter}}$}\label{fig.1}
	\end{minipage}	
	\begin{minipage}[t]{0.46 \linewidth}
		\centering
		\includegraphics[scale=0.5]{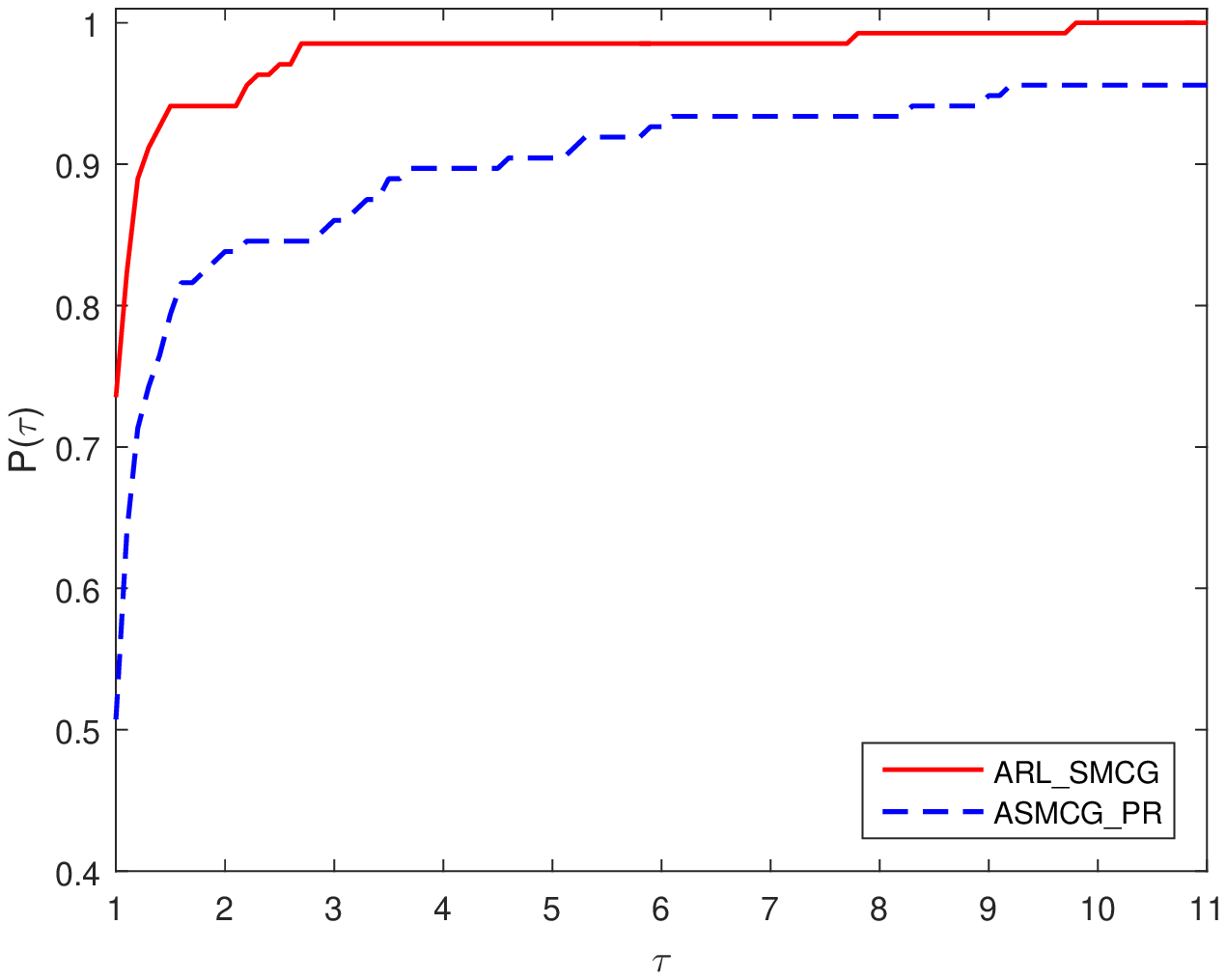}
%		\caption{Performance profile based on ${N_{f}}$(CUTEr).}\label{fig.2}
\caption{ ${N_{f}}$ }\label{fig.2}
	\end{minipage}	
\end{figure}

\begin{figure}[htp]
	\centering
	\begin{minipage}[t]{0.46 \linewidth}
		\centering
		\includegraphics[scale=0.5]{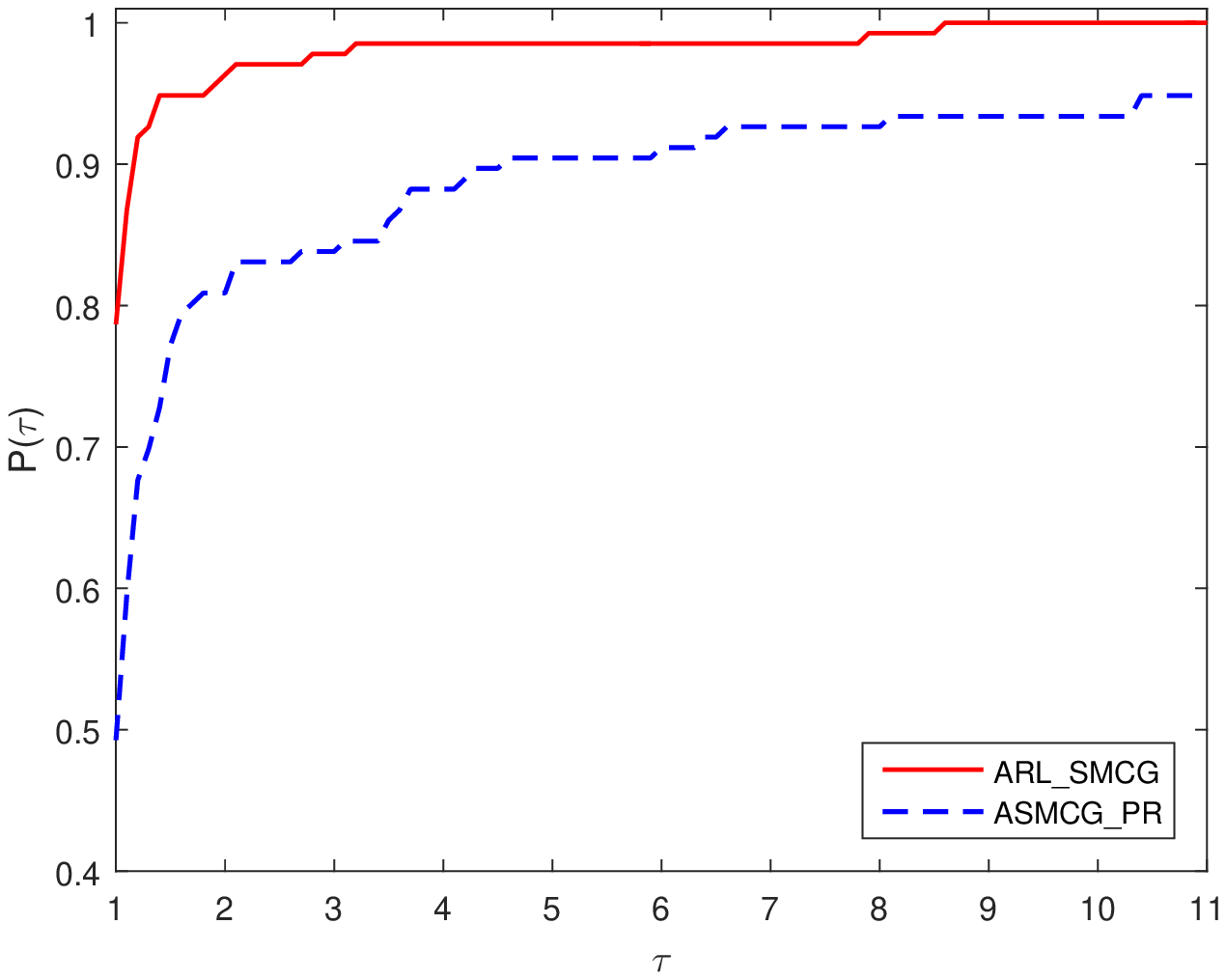}
%		\caption{Performance profile based on ${N_{g}}$(CUTEr).}\label{fig.3}
\caption{ ${N_{g}}$ }\label{fig.3}
	\end{minipage}	
	\begin{minipage}[t]{0.46 \linewidth}
		\centering
		\includegraphics[scale=0.5]{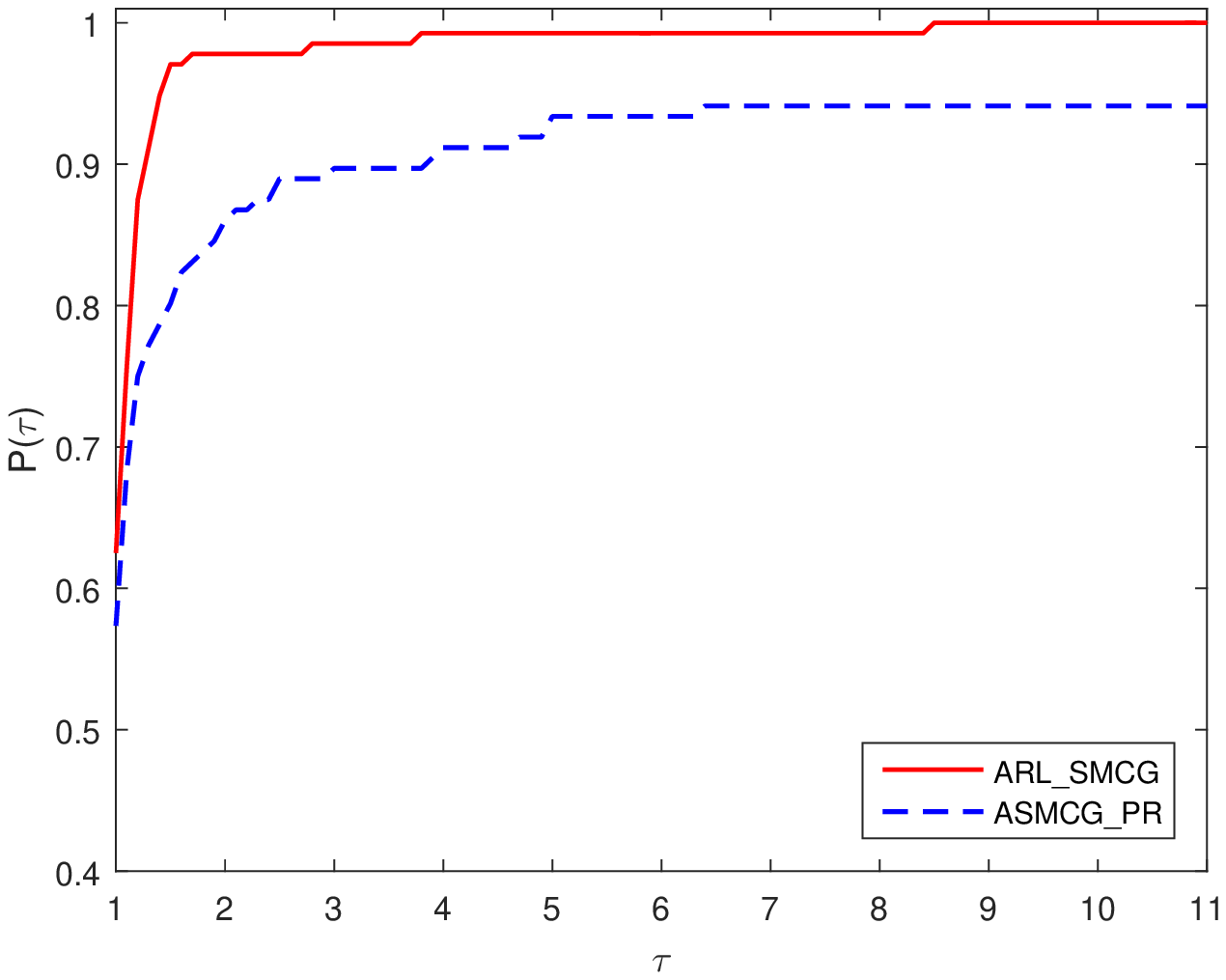}
%		\caption{Performance profile based on ${T_{cpu}}$(CUTEr).}\label{fig.4}
\caption{ ${T_{cpu}}$ }\label{fig.4}
	\end{minipage}	
\end{figure}

In the first set of numerical experiments, figures  \ref{fig.1}-\ref{fig.4} illustrate the performance profiles of RL\_SMCG and ASMCG\_PR \cite{Sun21}.  From Figs. \ref{fig.1}, \ref{fig.2}, \ref{fig.3} and \ref{fig.4}, we can observe that RL\_SMCG has a quite significant improvement over  ASMCG\_PR in terms of the number of iterations, the number of function evaluations, the number of gradient evaluations and CPU time. It indicates that the limited memory technique equipped in RL\_SMCG indeed brings quite significant numerical improvements.

In the second set of numerical experiments, we give a comparison of the performance profiles of RL\_SMCG with CG\_DESCENT(6.8) \cite{Hager2013}. Regarding the number of iterations and  the number of function evaluations in Fig. \ref{fig.5} and Fig. \ref{fig.6} respectively, we observe that RL\_SMCG is a little  better than CG\_DESCENT(6.8) for the number of iterations and the number of function evaluations.  As shown in Fig. \ref{fig.7}, we can see that RL\_SMCG is much better than CG\_DESCENT(6.8) in terms of the number of gradient evaluations,  because RL\_SMCG outperforms for about $71.5\% $ of the CUTEr test problems, while the percentage of software CG\_DESCENT(6.8) is below $40\%.$  It can be observe from Fig. \ref{fig.8} that RL\_SMCG is faster than CG\_DESCENT(6.8) in terms of CPU time.
By Theorem \ref{Th1}, RL\_SMCG is globally convergent with the generalized nonmonotone Wolfe line search, while CG\_DESCENT (6.8) does not guarantee global convergence when using the rather efficient approximate Wolfe (AWolfe) line search. This means that RL\_SMCG is superior to CG\_DESCENT(6.8) for CUTEr library in theory and numerical performance.

\begin{figure}[htp]
	\centering
	\begin{minipage}[t]{0.46 \linewidth}
		\includegraphics[scale=0.5]{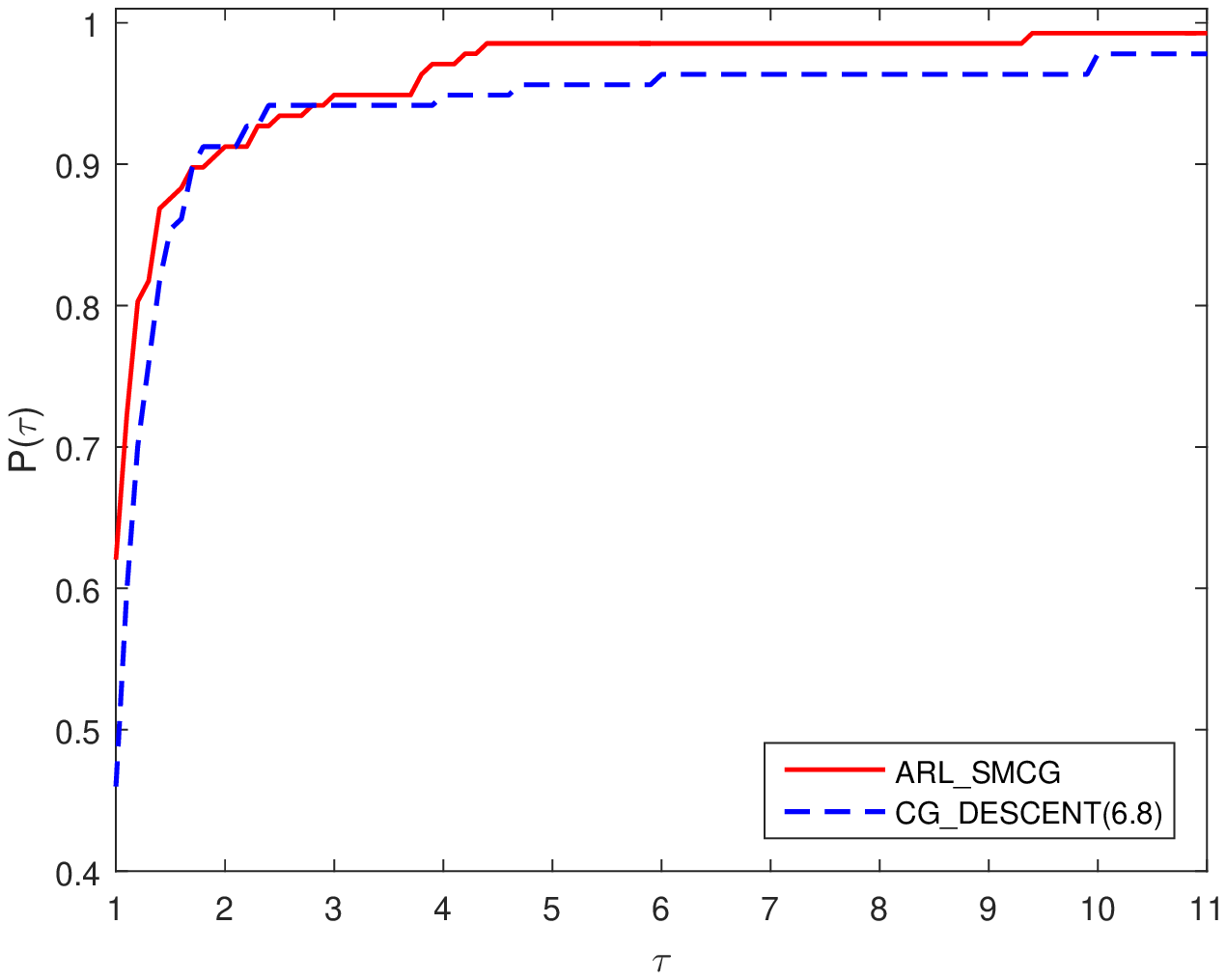}
%		\caption{Performance profile based on ${N_{iter}}$(CUTEr).}\label{fig.5}
\caption{ ${N_{iter}}$ }\label{fig.5}
	\end{minipage}	
	\begin{minipage}[t]{0.46 \linewidth}
		\centering
		\includegraphics[scale=0.5]{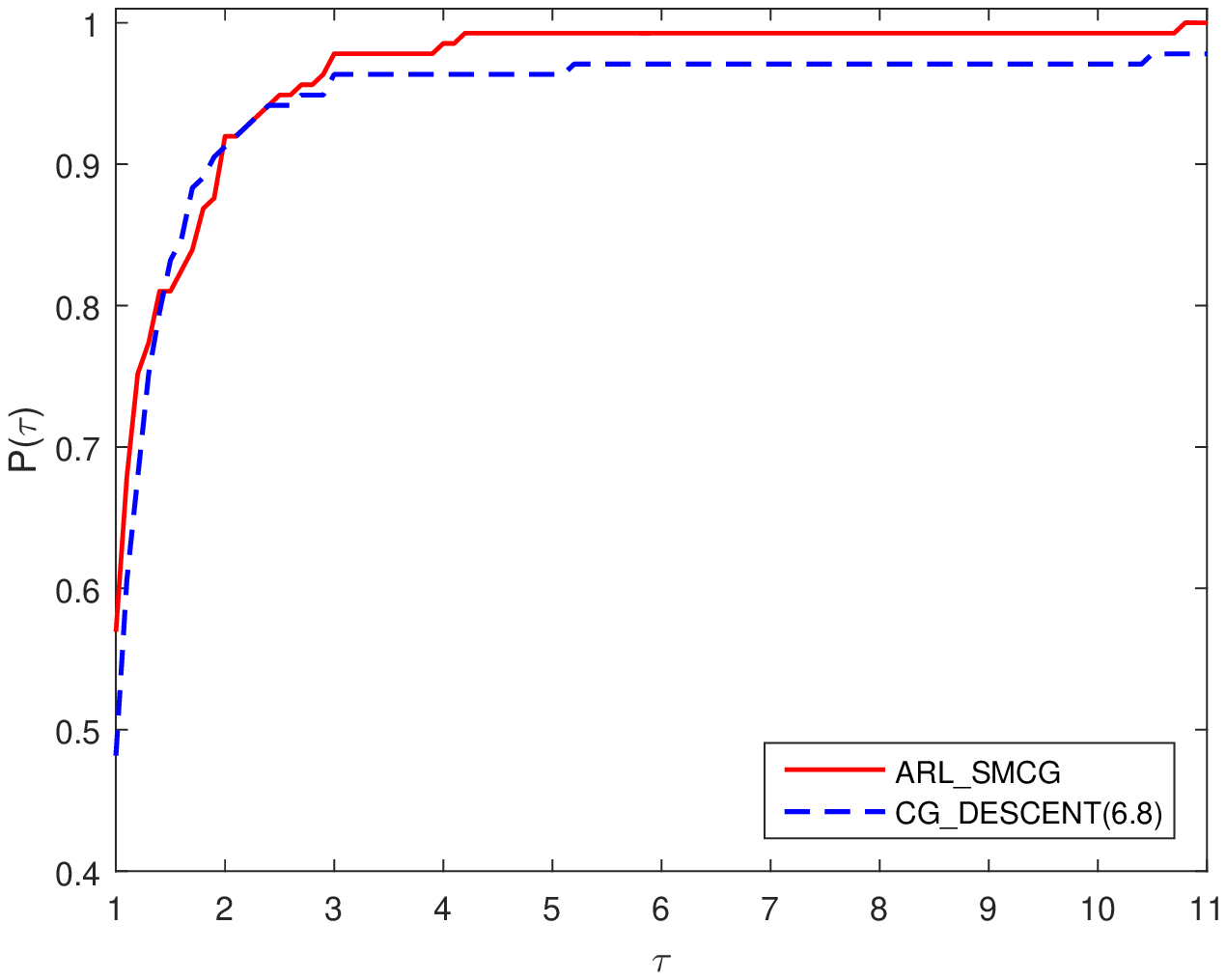}
%		\caption{Performance profile based on ${N_{f}}$(CUTEr).}\label{fig.6}
\caption{${N_{f}}$}\label{fig.6}
	\end{minipage}	
\end{figure}

\begin{figure}[htp]
	\centering
	\begin{minipage}[t]{0.46 \linewidth}
		\centering
		\includegraphics[scale=0.5]{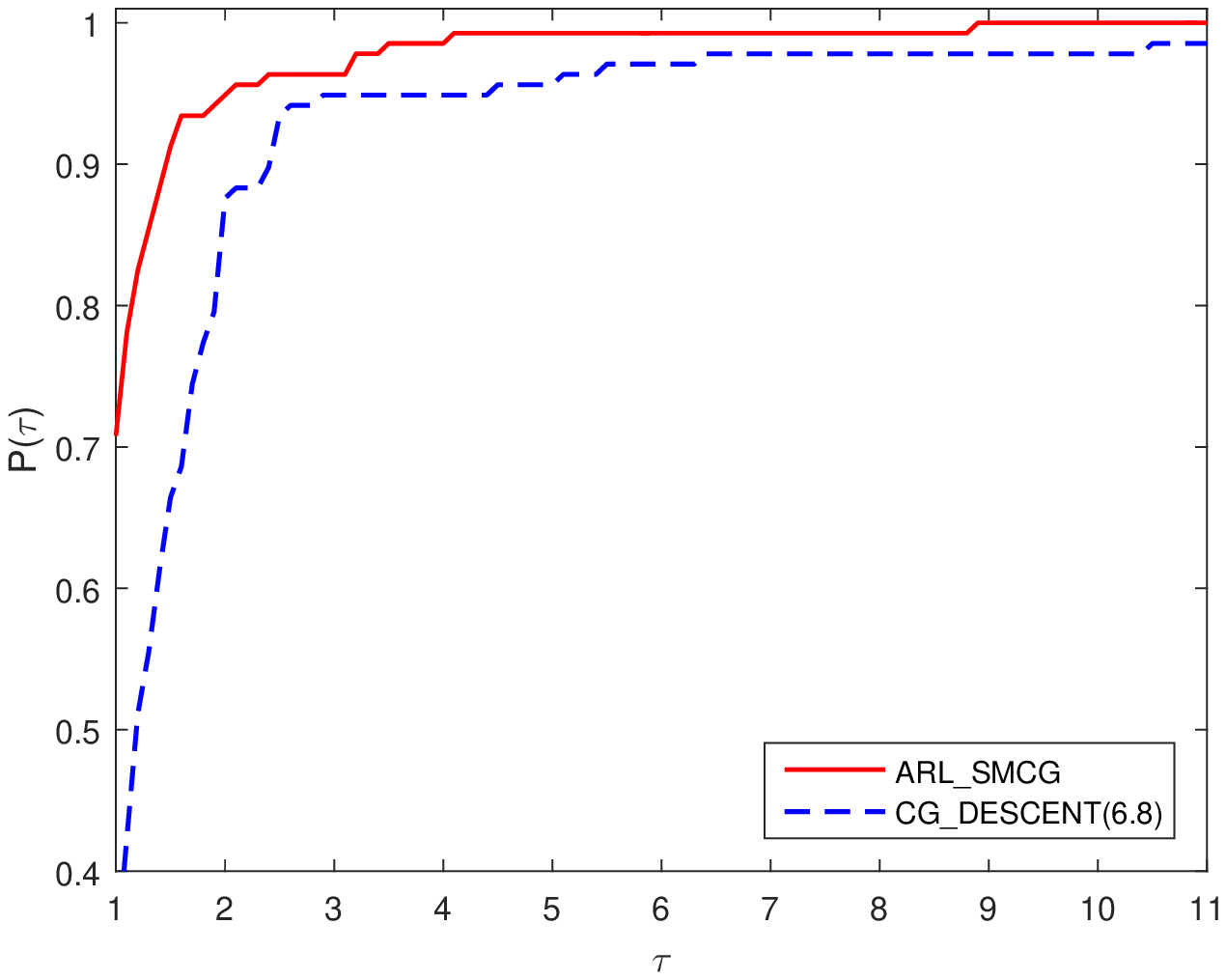}
%		\caption{Performance profile based on ${N_{g}}$(CUTEr).}\label{fig.7}
\caption{ ${N_{g}}$ }\label{fig.7}
	\end{minipage}	
	\begin{minipage}[t]{0.46 \linewidth}
		\centering
		\includegraphics[scale=0.5]{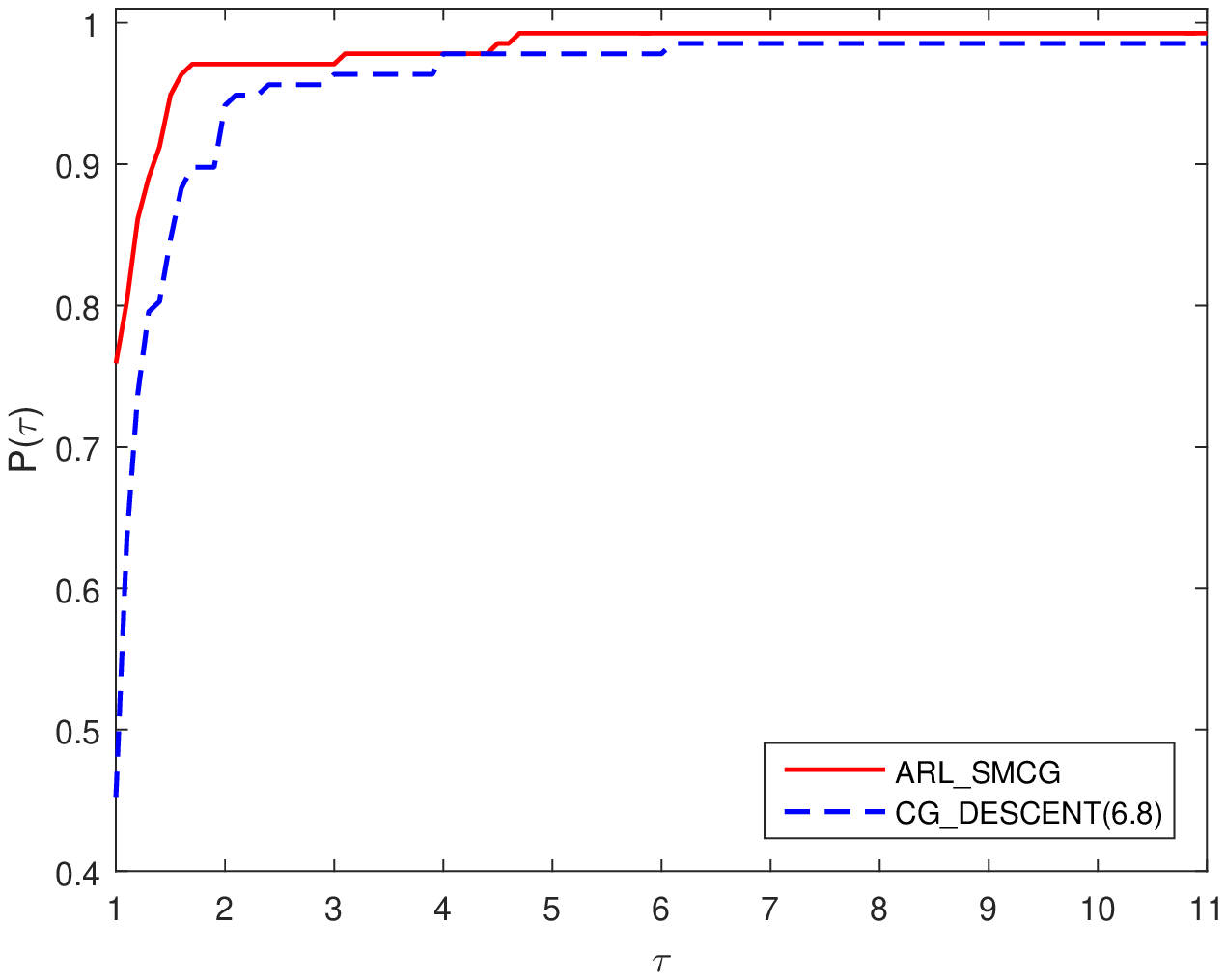}
%		\caption{Performance profile based on ${T_{cpu}}$(CUTEr).}\label{fig.8}
\caption{ ${T_{cpu}}$ }\label{fig.8}
	\end{minipage}	
\end{figure}

In the third set of the numerical experiments, comparing the performance of RL\_SMCG with CGOPT(2.0) \cite{Liu20}.  As shown in Figs. \ref{fig.9} and \ref{fig.10}, we can take a look at RL\_SMCG performs almost always  better than CGOPT(2.0)  in terms of the number of iterations and the number of function evaluations.  Figures. \ref{fig.11} and \ref{fig.12}  indicates that RL\_SMCG outperforms CGOPT(2.0) in terms of the number of gradient evaluations and CPU time  for the CUTEr library.

From the results of the above three numerical experiments, it is clear that the proposed algorithm RL\_SMCG is quite effective.

\begin{figure}[htp]
	\centering
	\begin{minipage}[t]{0.46 \linewidth}
		\includegraphics[scale=0.5]{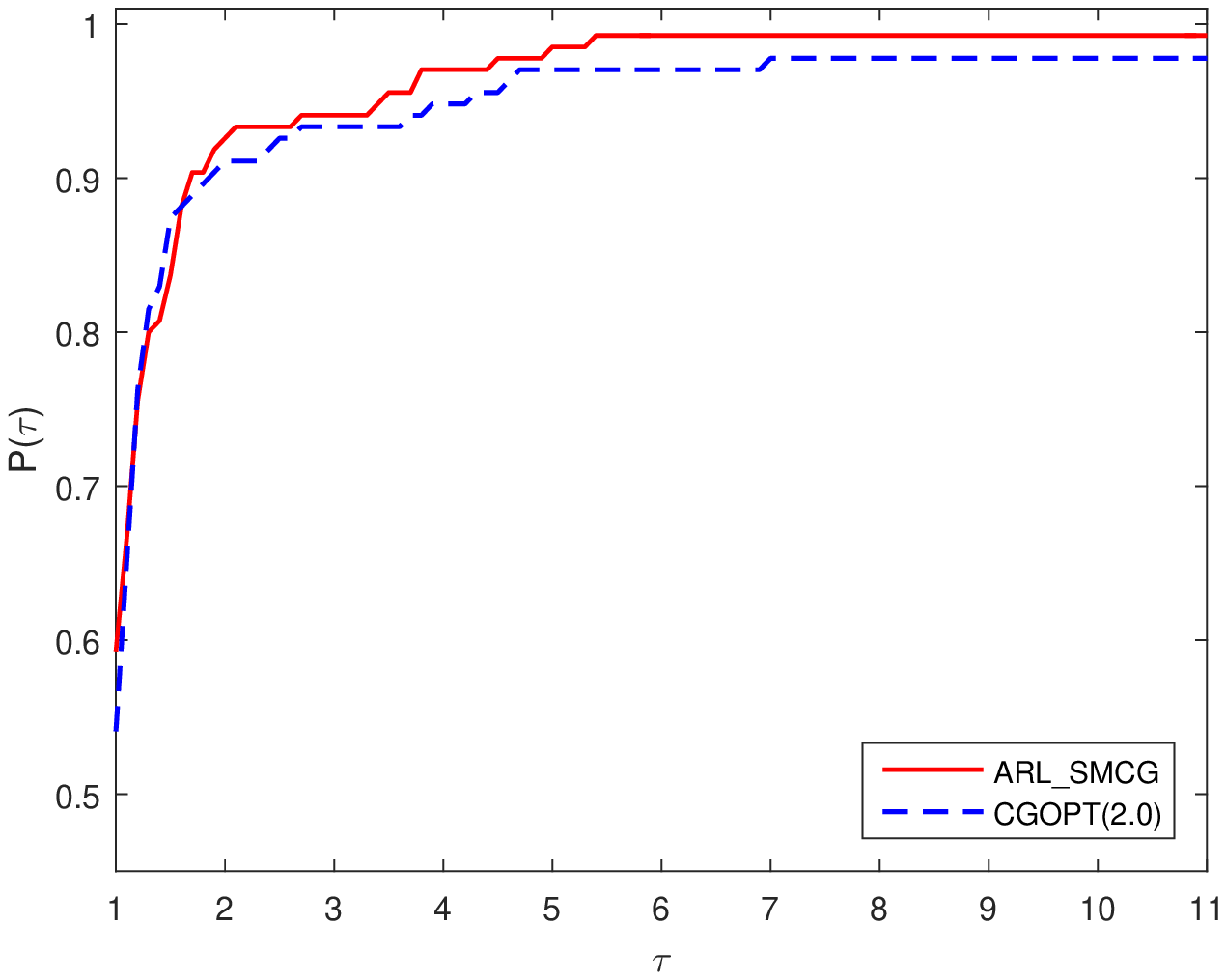}
%		\caption{Performance profile based on ${N_{iter}}$(CUTEr).}\label{fig.9}
\caption{ ${N_{iter}}$ }\label{fig.9}
	\end{minipage}	
	\begin{minipage}[t]{0.46 \linewidth}
		\centering
		\includegraphics[scale=0.5]{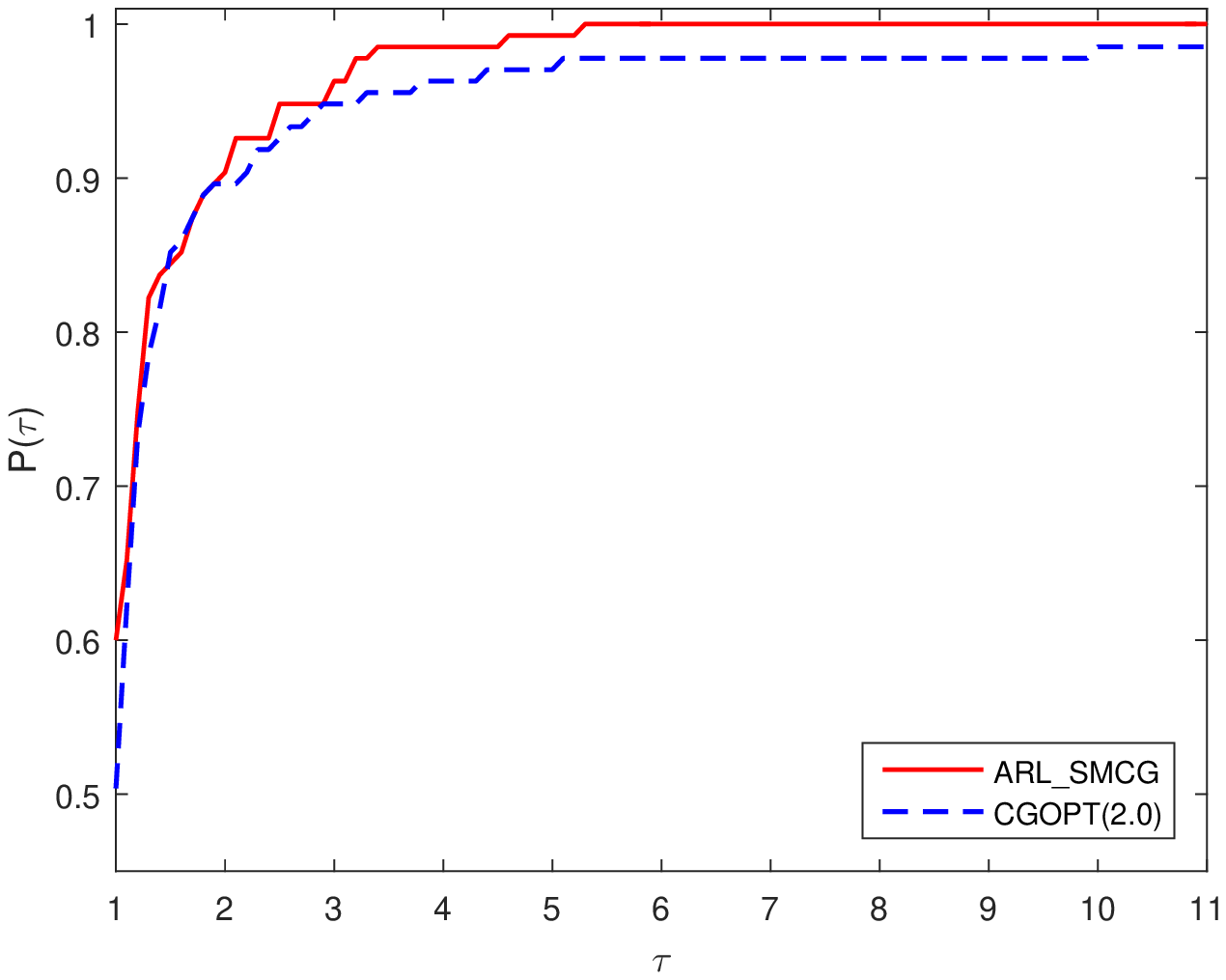}
%		\caption{Performance profile based on ${N_{f}}$(CUTEr).}\label{fig.10}
\caption{ ${N_{f}}$ }\label{fig.10}
	\end{minipage}	
\end{figure}

\begin{figure}[htp]
	\centering
	\begin{minipage}[t]{0.46 \linewidth}
		\centering
		\includegraphics[scale=0.5]{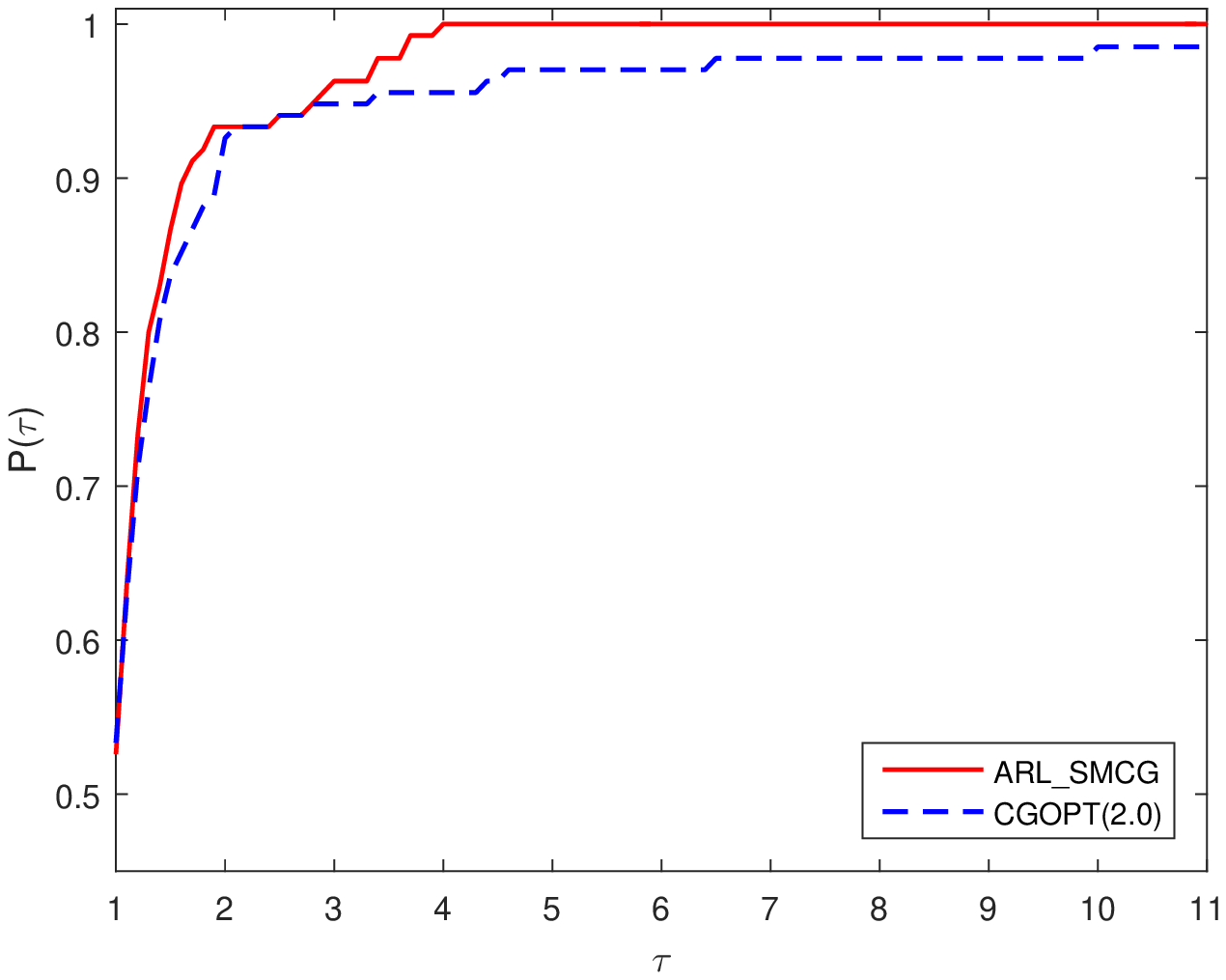}
%		\caption{Performance profile based on ${N_{g}}$(CUTEr).}\label{fig.11}
\caption{ ${N_{g}}$ }\label{fig.11}
	\end{minipage}	
	\begin{minipage}[t]{0.46 \linewidth}
		\centering
		\includegraphics[scale=0.5]{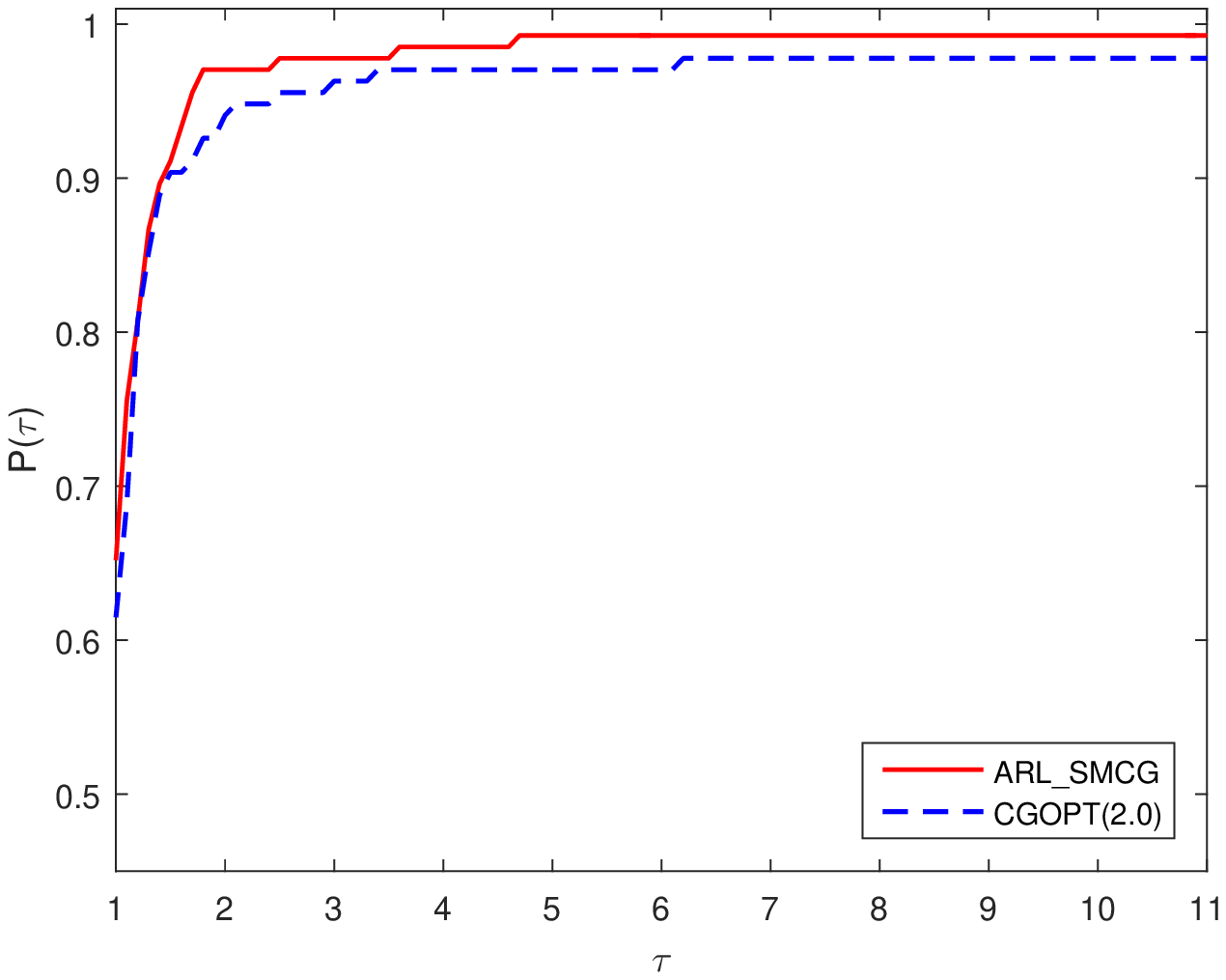}
%		\caption{Performance profile based on ${T_{cpu}}$(CUTEr).}\label{fig.12}
\caption{ ${T_{cpu}}$ }\label{fig.12}
	\end{minipage}	
\end{figure}

\section{Conclusions}\label{sec5}
In this paper, combined subspace minimization conjugate gradient method with limited memory technique, we presented a regularized limited memory subspace minimization conjugate gradient method, which contains two types of iteration. In the proposed algorithm, a modified regularized quasi-Newton method is given in small dimensional subspace to correct the orthogonality, and an improved initial step size selection strategy and some simple acceleration criteria are designed. Moreover, we establish the global convergence of the proposed algorithm by utilizing generalized nonmonotone  Wolfe line search under some mild assumptions. Some numerical results suggest that our algorithm  yields a tremendous improvement over the ASMCG\_PR and  outperforms  the  most up-to-date  limited memory CG software packages CG\_DESCENT (6.8) and CGOPT(2.0).

%\begin{acknowledgements}
%The authors would like to thank the editor and the anonymous referees for their valuable suggestions and comments which have greatly improved the presentation of this paper.
%This research was supported by the National Natural Science Foundation of China (No. 11901561), the Natural Science Foundation of Guizhou (No. ZK[2022]084) and the Natural Science Basic Research Program of Shaanxi (No. 2021JM-396).
%\end{acknowledgements}

\section{Declarations}
\subsection{Ethical Approval}
Not Applicable
\subsection{Availability of supporting data}
Data sharing not applicable to this article as no datasets were generated or analyzed during the current study.
\subsection{Competing interests}
The authors declare no competing interests.
\subsection{Funding}
This research was supported by the National Natural Science Foundation of China (No. 11901561), the Natural Science Foundation of Guizhou (No. ZK[2022]084) and the Natural Science Basic Research Program of Shaanxi (No. 2021JM-396).
\subsection{Authors' contributions}
Wumei Sun wrote the main manuscript text. Hongwei Liu and Zexian Liu reviewed and revised the manuscript.
\subsection{Acknowledgments}
The authors would like to thank the editor and the anonymous referees for their valuable suggestions and comments which have greatly improved the presentation of this paper.

\end{spacing}

\begin{thebibliography}{}
\expandafter\ifx\csname url\endcsname\relax
  \def\url#1{\texttt{#1}}\fi
\expandafter\ifx\csname urlprefix\endcsname\relax\def\urlprefix{URL }\fi
\expandafter\ifx\csname href\endcsname\relax
  \def\href#1#2{#2} \def\path#1{#1}\fi

\end{thebibliography}


\begin{thebibliography}{00}
 \bibitem{Andrei14} Andrei, N.: An accelerated subspace minimization three-term conjugate gradient algorithm for unconstrained optimization. Numer. Algor. \textbf{65}, 859-874 (2014)

 \bibitem{Barzilai88} Barzilai, J., Borwein, J.M.: Two-point step size gradient methods. IMA J. Numer Anal. \textbf{8}, 141-148 (1988)

 \bibitem{Dai02} Dai, Y.H., Yuan, J.Y., Yuan, Y.X.: Modified two-point stepsize gradient methods for unconstrained optimization problems. Comput. Optim. Appl. \textbf{22}(1), 103-109 (2002)

 \bibitem{Dai11} Dai, Y.H.: Nonlinear Conjugate Gradient Methods. Wiley Encyclopedia of Operations Research and Management Science(2011). https://doi.org/10.1002/9780470400531.eorms0183

 \bibitem{Dai13} Dai, Y.H., Kou, C.X.: A nonlinear conjugate gradient algorithm with an optimal property and an improved Wolfe line search. SIAM J. Optim. \textbf{23}(1), 296-320 (2013)

 \bibitem{Dai16} Dai, Y.H., Kou, C.X.: A Barzilai-Borwein conjugate gradient method. Sci. China Math. \textbf{59}(8), 1511-1524 (2016)

 \bibitem{Dai99} Dai, Y.H., Yuan, Y.: A nonlinear conjugate gradient method with a strong global convergence property. SIAM J. Optim. \textbf{10}(1), 177-182 (1999)

 \bibitem{Dolan02} Dolan, E.D., Mor$\acute{\text{e}}$, J.J.: Benchmarking optimization software with performance profiles. Math.   Program. \textbf{91}, 201-213 (2002)

 \bibitem{Fletcher64} Fletcher, R., Reeves, C.M.: Function minimization by conjugate gradients. Computer Journal. \textbf{7}, 149-154 (1964)

 \bibitem{Gould03} Gould, N.I.M., Orban, D., Toint, Ph.L: CUTEr and SifDec: A Constrained and Unconstrained Testing Environment, revisited. ACM Trans. Math. Softw. \textbf{29}, 373-394 (2003)

 \bibitem{Gu2003} Gu, G.Z., Li, D.H., Qi, L.Q., Zhou, S.Z.: Descent directions of quasi-Newton methods for symmetric nonlinear equations. SIAM J. Numer. Anal. \textbf{40}, 1763-1774 (2003)

 \bibitem{Hager05} Hager, W.W., Zhang, H.: A new conjugate gradient method with guaranteed descent and an efficient line search. SIAM J. Optim. \textbf{16}(1), 170-192 (2005)

 \bibitem{Hager06a} Hager, W.W., Zhang, H.: A survey of nonlinear conjugate gradient methods. Pac. J. Optim. \textbf{2}(1), 35-58 (2006)
 \bibitem{Hager06b} Hager, W.W., Zhang, H.: Algorithm 851: CG\_DESCENT, a conjugate gradient method with guaranteed descent. ACM Trans. Math. Software. \textbf{32}(1), 113-137 (2006)

 \bibitem{Hager2013} Hager, W.W., Zhang, H.: The limited memory conjugate gradient method. SIAM J. Optim. \textbf{23}, 2150-2168 (2013)

 \bibitem{Hestenes52} Hestenes, M.R., Stiefel, E.: Methods of conjugate gradients for solving linear systems. J. Res. Natl. Bur. Stand. \textbf{49}, 409-436 (1952)

 \bibitem{Huang2015} Huang, S., Wan, Z., Chen, X.H.: A new nonmonotone line search technique for unconstrained optimization. Numer. Algor. \textbf{68}(4), 671-689 (2015)

 \bibitem{Li1999} Li, D.H., Fukushima, M.: A globally and superlinearly convergent Gauss-Newton-based BFGS methods for symmetric nonlinear equations. SIAM J. Numer. Anal. \textbf{37}, 152-172 (1999)

 \bibitem{Li01}  Li, D. H., Fukushima, M.: On the global convergence of BFGS method for nonconvex unconstrained optimization problems. SIAM J. Optim. \textbf{11}(4), 1054-1064 (2001)

 \bibitem{Li18} Li, M., Liu, H.W., Liu, Z.X.: A new subspace minimization conjugate gradient method with nonmonotone line search for unconstrained optimization. Numer Algor.  \textbf{79}, 195-219 (2018)

 \bibitem{Li19} Li, Y.F., Liu, Z.X., Liu, H.W.: A subspace minimization conjugate gradient method based on conic model for unconstrained optimization. Computational and Applied Mathematics. \textbf{38}(1), (2019)

 \bibitem{Liu1989} Liu, D.C., Nocedal, J.: On the limited memory BFGS method for large scale optimization. Math. Program. \textbf{45}, 503-528 (1989)

 \bibitem{Liu2014} Liu, T. W.: A regularized limited memory BFGS method for nonconvex unconstrained minimization. Numer. Algor. \textbf{65}, 305-323 (2014)

 \bibitem{Liu18} Liu, Z.X., Liu, H.W.: An efficient gradient method with approximate optimal stepsize for large-scale unconstrained optimization. Numer. Algorithms \textbf{78}(1), 21-39 (2018)

 \bibitem{Liu18b}  Liu, Z.X., Liu, H.W.: Several efficient gradient methods with approximate optimal stepsizes for large scale unconstrained optimization. J. Comput. Appl. Math. 328, 400-413 (2018)

 \bibitem{Liu19} Liu, H.W., Liu, Z.X.: An efficient Barzilai-Borwein conjugate gradient method for unconstrained optimization. J. Optim. Theory Appl. \textbf{180}, 879-906 (2019)

 \bibitem{Liu20} Liu, Z.X., Liu, H.W., Dai, Y.H.: An improved Dai¨CKou conjugate gradient algorithm for unconstrained optimization. Comput. Optim. Appl. \textbf{75}(1), 145-167 (2020)

 \bibitem{Nocedal1980}  Nocedal, J.: Updating quasi-Newton matrices with limited storage. Math. Comput. \textbf{35}, 773-782 (1980)

 \bibitem{Nocedal99} Nocedal, J., Wright, S.J.: Numerical Optimization. New York, Springer (1999)

 \bibitem{Polak69a} Polak, E., Ribi$\rm \grave{e}$re, G.: Note sur la convergence de m$\rm \acute{e}$thodes de directions conjugu$\rm \acute{e}$es. Rev. Franaise Informat. Rech. Op$\rm \acute{e}$rationnelle. \textbf{3}(16), 35-43 (1969)

 \bibitem{Polyak69b} Polyak, B.T.: The conjugate gradient method in extremal problems. Ussr Comput. Math. Math. Phys. \textbf{9}(4), 94-112 (1969)

 \bibitem{Sun21} Sun, W., Liu, H., Liu, Z.: A Class of Accelerated Subspace Minimization Conjugate Gradient Methods. J. Optim. Theory Appl. \textbf{190}(3), 811-840 (2021)

 \bibitem{Tarzangh2015} Tarzangh, D.A., Peyghami, M.R.: A new regularized limited memory BFGS-type method based on modified secant conditions for unconstrained optimization problems. J. Global Optim. \textbf{63}, 709-728 (2015)

 \bibitem{Tankaria2022} Tankaria, H., Sugimoto, S.,  Yamashita, N.: A regularized limited memory BFGS method for large-scale unconstrained optimization and its efficient implementations. Comput. Optim. Appl. \textbf{82}, 61-88 (2022)

 \bibitem{Ueda2010}  Ueda, K., Yamashita, N.: Convergence properties of the regularized newton method for the unconstrained nonconvex optimization. Appl. Math. Optim. \textbf{62}, 27-46 (2010)

 \bibitem{Wang19} Wang, T., Liu, Z.X., Liu, H.W.: A new subspace minimization conjugate gradient method based on tensor model for unconstrained optimization. Int. J. Comput. Math. \textbf{96}(10), 1924-1942 (2019)

 \bibitem{Yang17} Yang, Y.T., Chen, Y.T. Lu, Y.L.: A subspace conjugate gradient algorithm for large-scale unconstrained optimization. Numer Algor. \textbf{76}, 813-828 (2017)

 \bibitem{Yuan91}  Yuan, Y.X.: A modified BFGS algorithm for unconstrained optimization. IMA J. Numer. Anal. \textbf{11}(3), 325-332 (1991)

 \bibitem{Yuan95} Yuan, Y.X., Stoer, J.: A subspace study on conjugate gradient algorithms. Z. Angew. Math. Mech. \textbf{75}(1), 69-77 (1995)

 \bibitem{Yuan99}  Yuan, Y. X., Sun, W. Y.: Theory and methods of optimization. Science Press of China (1999)

 \bibitem{Zhang04} Zhang, H., Hager, W.W.,: A Nonmonotone Line Search Technique and Its Application to Unconstrained Optimization. SIAM J. Optim. \textbf{14}(4), 1043-1056 (2004)

 \bibitem{Zhao21} Zhao, T., Liu, H.W., Liu, Z.X.: New subspace minimization conjugate gradient methods based on regularization model for unconstrained optimization. Numer. Algor. \textbf{87}, 1501-1534 (2021)

\end{thebibliography}
\end{document}